\documentclass[11pt]{amsart}
\usepackage{color}
\usepackage{graphicx}
\usepackage{hyperref}
\usepackage{amssymb}
\usepackage{amsfonts}
\usepackage{euscript}
\usepackage[matrix,arrow,curve]{xy}

\numberwithin{equation}{section}
\renewcommand{\subsection}[1]{\hspace{-\parindent}\refstepcounter{subsection}{\bf (\arabic{section}.\alph{subsection}) #1.}\addcontentsline{toc}{subsection}{\bf #1.}}

\newenvironment{nouppercase}{%
  \renewcommand{\uppercasenonmath}[1]{}}{}

\newtheorem{thm}{Theorem}[section]
\newtheorem{theorem}[thm]{Theorem}

\newtheorem{corollary}[thm]{Corollary}

\newtheorem{definition}[thm]{Definition}

\newtheorem{remark}[thm]{Remark}
\newtheorem{convention}[thm]{Convention}
\newtheorem{proposition}[thm]{Proposition}
\newtheorem{example}[thm]{Example}

\newtheorem{lemma}[thm]{Lemma}
\newtheorem{setup}[thm]{Setup}

\newtheorem*{claim*}{Claim} 
\newtheorem*{lemma*}{Lemma}

\newcommand{\bC}{{\mathbb C}}

\newcommand{\bQ}{{\mathbb Q}}
\newcommand{\bR}{{\mathbb R}}

\newcommand{\bZ}{{\mathbb Z}}

\newcommand{\scrA}{\EuScript A}
\newcommand{\scrB}{\EuScript B}
\newcommand{\scrC}{\EuScript C}

\newcommand{\scrF}{\EuScript F}

\newcommand{\scrI}{\EuScript I}

\newcommand{\scrK}{\EuScript K}

\newcommand{\scrM}{\EuScript M}

\newcommand{\scrT}{\EuScript T}

\newcommand{\scrX}{\EuScript X}
\newcommand{\scrY}{\EuScript Y}

\newcommand{\iso}{\cong}
\newcommand{\htp}{\simeq}

\renewcommand{\hom}{\mathit{hom}}
\newcommand{\Ob}{\mathit{Ob}}

\newcommand{\Tw}{\mathit{Tw}}
\newcommand{\Mod}{\mathit{Mod}}
\newcommand{\EqTw}{\mathit{EqTw}}
\newcommand{\EqMod}{\mathit{EqMod}}

\title[Picard-Lefschetz theory]{\Large\larger\rm Picard-Lefschetz theory and dilating $\bC^*$-actions}
\author{Paul Seidel}

\begin{document}
\begin{nouppercase}
\maketitle
\end{nouppercase}

\begin{abstract}
We consider $\bC^*$-actions on Fukaya categories of exact symplectic manifolds. Such actions can be constructed by dimensional induction, going from the fibre of a Lefschetz fibration to its total space. We explore applications to the topology of Lagrangian submanifolds, with an emphasis on ease of computation.
\end{abstract}

\setcounter{section}{-1}
\section{Introduction\label{sec:intro}}

It has been gradually recognized that certain classes of non-closed symplectic manifolds admit symmetries of a new kind. These symmetries are not given by groups acting on the manifold, but instead appear as extra structure on Floer cohomology, or more properly on the Fukaya category. The first example may have been the bigrading on the Floer cohomology of certain Lagrangian spheres in the Milnor fibres of type $(A)$ hypersurface singularities in $\bC^n$, described in \cite{khovanov-seidel98} (it turned out later \cite[Section 20]{seidel04} that this is compatible with the $A_\infty$-structure of the Fukaya category only if $n \geq 3$). The geometric origin of such symmetries (on the infinitesimal level) has been studied in \cite{seidel-solomon10}, with applications in \cite{seidel13, lekili-pascaleff13, abouzaid-smith13}. A roadmap is provided (via mirror symmetry) by the theory of equivariant coherent sheaves, and its applications in algebraic geometry and geometric representation theory; the relevant literature is too vast to survey properly, but \cite{polishchuk08, elagin12} have been influential for the developments presented here. 

In \cite{seidel12}, the example from \cite{khovanov-seidel98} was used as a test case for talking about such symmetries in an algebraic language of $A_\infty$-categories with $\bC^*$-actions. Here, we generalize that approach, and combine it with the symplectic version of Picard-Lefschetz theory \cite{seidel04}. Recall that classical Picard-Lefschetz theory provides (among other things) a way of computing the intersection pairing in the middle-dimensional homology of an affine algebraic variety, by induction on the dimension. As one consequence of our construction, one gets a similar machinery for defining and computing algebraic analogues of the $q$-intersection numbers from \cite{seidel-solomon10}. Aside from their intrinsic interest, these $q$-intersection numbers have implications for the topology of Lagrangian submanifolds. These are similar in spirit to those derived in \cite{seidel13}, but benefit from the more rigid setup of $\bC^*$-actions, as well as the greater ease of doing computations in a purely algebraic framework. In particular, Example \ref{th:q-mirrorp2} would be out of reach of the methods in \cite{seidel13}, since those only involved the first derivative in the equivariant parameter $q$ around $q = 1$, whereas \eqref{eq:determinant} vanishes at least to order $2$ at that point. Another noteworthy comparison is \cite{lekili-maydanskiy12}, which contains an example of a non-existence theorem for Lagrangian submanifolds whose statement is of a similar kind to that in Example \ref{th:q-mirrorp2}. However, the proof in \cite{lekili-maydanskiy12} relies on a classification of spherical objects \cite{ishii-ueda-uehara10}, which is difficult to generalize to other situations.

There is a foundational question here, which is what one really means by ``the Fukaya category admits a $\bC^*$-action''. Here, we use a simple ad hoc definition, which is not intrinsic: it amounts to saying that the Fukaya category can be embedded into a larger category, and that the larger category carries a $\bC^*$-action in a more naive sense. Certain questions, such as the persistence of the symmetry under deformations of the symplectic form, cannot be meaningfully addressed in this framework. However, at present these are mostly theoretical shortcomings: while there are definitions of an action of $\bC^*$ (or more general algebraic group) on an $A_\infty$-category which are more satisfying, they have not yet led to substantially stronger applications in symplectic topology. 

The structure of this paper is as follows:
\begin{itemize}\itemsep1em 
\item In Section \ref{sec:review}, we focus on the $q$-intersection numbers and their applications, leaving the categorical structures out of the picture as much as possible. This allows us to present the theory as a natural extension of classical Picard-Lefschetz theory, even though certain notions must remain temporarily unexplained. Several concrete examples are considered.

\item Sections \ref{sec:background} and \ref{sec:tools} collect the necessary notions from homological algebra. All of them are quite basic, except for a construction of equivariant $A_\infty$-modules which we quote from \cite{seidel12}. The main feature of the exposition is that directed $A_\infty$-categories play a preferred role.

\item Section \ref{sec:fukaya} applies the algebraic theory to Fukaya categories. After a review of some material about Lefschetz fibrations from \cite{seidel04}, we derive the general results stated in Section \ref{sec:review} (the core of the argument is in Section \ref{subsec:proofs}), and also complete the computations required in our examples.
\end{itemize}

{\em Acknowledgments.} I owe a substantial debt to Mohammed Abouzaid, who suggested that I should look at Lefschetz fibrations whose fibres admit $\bC^*$-actions. Partial support was provided by NSF grants DMS-1005288 and DMS-1500954, and by the Simons Foundation through a Simons Investigator grant. I would also like to thank Boston College, where a large part of the paper was written, for its hospitality.
\section{$q$-intersection numbers\label{sec:review}}

\subsection{Classical Picard-Lefschetz theory\label{subsec:classical}}
Picard-Lefschetz theory provides a way of analyzing the topology of a class of manifolds (complex algebraic varieties, as well as certain complex analytic and symplectic manifolds). We will be particularly interested in what the theory has to say about the middle-dimensional homology and its intersection pairing.

\begin{setup} \label{th:topological-lefschetz}
We will work with smooth Lefschetz fibrations
\begin{equation} \label{eq:lefschetz}
\pi: E \longrightarrow D,
\end{equation}
where the base $D$ is a closed oriented disc, and the total space $E$ an oriented compact $2n$-manifold with corners. The local behaviour of $\pi$ is as follows:
\begin{itemize}
\itemsep0.5em
\item Near a regular point, the local model for \eqref{eq:lefschetz} is a trivial fibration with fibre  $\bR^{2n-2-k} \times (\bR^+)^k$ for some $k$ ($k = 0$ is a fibrewise interior point, $k>0$ a fibrewise boundary point);
\item Critical points must lie in the interior of $E$, and their image must lie in the interior of $D$. At most one such point may lie in each fibre. Finally, around each critical point, there should be oriented complex coordinates on $E$ and $D$, in which $\pi(x_1,\dots,x_n) = x_1^2 + \cdots + x_n^2$.
\end{itemize}
We fix a base point $\ast \in \partial D$, and denote the fibre over that point by $F = E_\ast$. This is a $(2n-2)$-dimensional manifold with corners.
\end{setup}

Away from the critical points, one can equip the fibration with a horizontal distribution (a subspace of the tangent space complementary to $\mathit{ker}(D\pi) \subset TE$) which is parallel to all the boundary strata of the fibres. This provides parallel transport maps, away from the singular fibres.

\begin{remark}
In most algebro-geometric applications, one starts with an open algebraic variety and a regular function on it, satisfying suitable conditions. The setup \eqref{eq:lefschetz} is then recovered by passing to a (carefully chosen) large compact subset. When talking about Picard-Lefschetz theory for algebraic varieties, we will implicitly assume that this transition has being carried out.
\end{remark}

Up to homotopy equivalence, $E$ is obtained from $F$ by attaching $n$-cells, one for each critical point of $\pi$. In particular, the abelian group 
\begin{equation}
H_\pi = H_n(E,F)
\end{equation}
is free, and its rank equals the number of critical points, which we will denote by $m$. $H_\pi$ carries a modified version of the intersection pairing, which we call the variation pairing, with notation
\begin{equation} \label{eq:variation-pairing}
(h_0,h_1) \longmapsto h_0 \cdot_\pi h_1 \in \bZ.
\end{equation}
To construct that, one moves $\ast$ in positive direction along the boundary by a small amount, to a new point $\tilde{\ast}$. One can realize this by an isotopy of the disc, and then lift that to an isotopy of $E$, by parallel transport. The endpoint of that isotopy is a diffeomorphism which maps $F$ to the fibre $\tilde{F}$ over $\tilde{\ast}$. Denote by $\tilde{h}$ the image of a homology class $h$ under the resulting isomorphism $H_*(E,F) \iso H_*(E,\tilde{F})$. One defines
\begin{equation}
h_0 \cdot_\pi h_1 = \tilde{h}_0 \cdot h_1,
\end{equation}
where the right hand side is the standard intersection pairing between $H_n(E,\tilde{F})$ and $H_n(E,F)$ (which makes sense because $F \cap \tilde{F} = \emptyset$). The variation pairing has the following properties:
\begin{itemize} \itemsep0.5em
\item On classes in the image of $H_n(E) \rightarrow H_\pi$, it equals the standard intersection pairing. 

\item It is not (graded) symmetric, but instead satisfies
\begin{equation} \label{eq:symmetrized}
h_0 \cdot_\pi h_1 - (-1)^n \, h_1 \cdot_\pi h_0 = \partial h_0 \cdot \partial h_1,
\end{equation}
where $\partial: H_\pi \rightarrow H_{n-1}(F)$ is the boundary map, and the right hand side is the intersection pairing in $F$.
\end{itemize}

\begin{convention} \label{th:convention-dot}
Our definition of the intersection number of middle-dimensional cycles on a $2n$-manifold differs from the usual one by a sign $(-1)^{n(n+1)/2}$. As a consequence, if the ambient manifold has an almost complex structure, and $L$ is an oriented closed totally real submanifold, its selfintersection number equals the Euler characteristic:
\begin{equation} \label{eq:chi}
L \cdot L = \chi(L).
\end{equation}
One can view this as a change of convention for the orientation of an almost complex manifold. If $(\xi_1,\dots,\xi_n)$ is an oriented basis of the tangent space of $L$, the usual orientation of the ambient almost complex manifold is given by $(\xi_1,J\xi_1,\dots,\xi_n,J\xi_n)$. Changing it by $(-1)^{n(n+1)/2}$ yields $(J\xi_1,\dots,J\xi_n,\xi_1,\dots,\xi_n)$, which agrees with the standard orientation of the total space of $TL$ in which the zero-section has self-intersection \eqref{eq:chi}.
\end{convention}

To make things more concrete, choose a basis (sometimes also called a distinguished set) of vanishing paths $(\gamma^1,\dots,\gamma^m)$; see \cite[Vol.\ 2, p.\ 14]{arnold-gusein-zade-varchenko} or \cite[Section 16]{seidel04}. To each $\gamma^k$ one can associate an embedded $n$-disc $\Delta^k \subset E$ (the Lefschetz thimble), with boundary $V^k = \partial \Delta^k \subset F$ (the vanishing cycle). These cycles satisfy (in accordance with \eqref{eq:symmetrized} as well as a suitably modified version of \eqref{eq:chi}, which we won't discuss in detail)
\begin{align}
& \Delta^k \cdot_\pi \Delta^k = \chi(\Delta^k) = 1, \\
& V^k \cdot V^k = \chi(V^k) = 1-(-1)^n. \label{eq:totally-real-sphere}
\end{align}
After choosing orientations, the $\Delta^k$ form a basis of $H_\pi$, which is such that the variation pairing is upper-triangular. More precisely, one has:
\begin{equation} \label{eq:v-matrix}
\Delta^i \cdot_\pi \Delta^j = \begin{cases} V^i \cdot V^j & i<j, \\ 1 &  i = j, \\ 0 & i>j.
\end{cases}
\end{equation}
To formulate this more concretely, let's introduce two matrices $A$, $B$ of size $m \times m$, with entries
\begin{align}
& A_{ij} = \Delta^i \cdot_\pi \Delta^j, \label{eq:a-matrix} \\
& B_{ij} = V^i \cdot V^j. \label{eq:b-matrix}
\end{align}
$B$ determines $A$ by \eqref{eq:v-matrix}; conversely, by \eqref{eq:totally-real-sphere} and the symmetry of the intersection pairing on $F$,
\begin{equation} \label{eq:ab-relation}
B = A - (-1)^n A^*,
\end{equation}
where $A^*$ is the transpose. If $L \subset E$ is an oriented closed $n$-dimensional submanifold, then its class $[L] \in H_\pi$ corresponds to a lattice element $l \in \bZ^m$ satisfying
\begin{equation}
B \, l = 0. \label{eq:h-condition-2}
\end{equation}
This just expresses the fact that $[L]$ must lie in the nullspace of \eqref{eq:symmetrized}, since $\partial [L] = 0$. If $L_0,L_1$ are two such submanifolds, then
\begin{equation} \label{eq:intersection-number}
l_0^* \, A \, l_1 = L_0 \cdot L_1.
\end{equation}
In view of \eqref{eq:v-matrix}, this in principle allows one to compute intersection pairings on middle-dimensional homology by dimensional induction (from the fibre to the total space of the Lefschetz fibration). 

Certain other invariants, derived from $A$, are also geometrically meaningful. Monodromy gives an automorphism $N$ of $H_\pi$, which is uniquely characterized by the fact that
\begin{equation} \label{eq:characterize-m}
h_0 \cdot_\pi N h_1 = (-1)^n h_1 \cdot_{\pi} h_0.
\end{equation}
In a basis of Lefschetz thimbles, this reduces to the formula \cite[vol.\ 2, Theorem 2.6]{arnold-gusein-zade-varchenko}
\begin{equation} \label{eq:monodromy}
N = (-1)^n\, A^{-1} A^* \in \mathit{GL}(m,\bZ).
\end{equation}
One can also consider a deformed version of \eqref{eq:ab-relation}, namely
\begin{equation} \label{eq:givental}
A - q(-1)^n A^* \in \mathit{Mat}(m \times m, \bZ[q,q^{-1}]),
\end{equation}
which defines a bilinear pairing on $H_\pi$ with values in $\bZ[q,q^{-1}]$. The geometric meaning of this pairing, in terms of homology with twisted coefficients, is described in \cite[Section 3]{givental88}. The determinant of \eqref{eq:givental} recovers the characteristic polynomial of $N$.

\subsection{Examples}
We work out \eqref{eq:a-matrix} and the implications of \eqref{eq:intersection-number} explicitly in three instances, of which the first one is very familiar (and included only because it reappears as an intermediate step in the other two). The results of these computations will be stated without detailed proofs; this is unproblematic since they are only intended as background for the more refined computations later on.

\begin{example} \label{th:an-milnor}
For any $r \geq 1$ and $n \geq 2$, consider the hypersurface $Y \subset \bC^n$ given (in coordinates $y,z_1,\dots,z_{n-1}$) by
\begin{equation} \label{eq:am}
Y = Y_r = \{p(y) = z_1^2 + \cdots + z_{n-1}^2\},
\end{equation}
where $p$ is a degree $r+1$ polynomial with only simple zeros. $Y$ is the Milnor fibre of the $(A_r)$ type hypersurface singularity in dimension $n-1$. The projection $y: Y \rightarrow \bC$ has only nondegenerate critical points (and will become a Lefschetz fibration, in the sense defined above, after passing to suitable compact subsets). Its critical values are precisely the zeros of $p$. Any embedded path in $\bC$ whose endpoints lie in $p^{-1}(0)$, and which otherwise avoids that set, defines an embedded totally real sphere $\Sigma \subset Y$ (see e.g.\ \cite{khovanov-seidel98}); one can think of $\Sigma$ as the union of two Lefschetz thimbles, associated to two halves of the path, joined along their common boundary.
\begin{figure}
\begin{centering}
\begin{picture}(0,0)%
\includegraphics{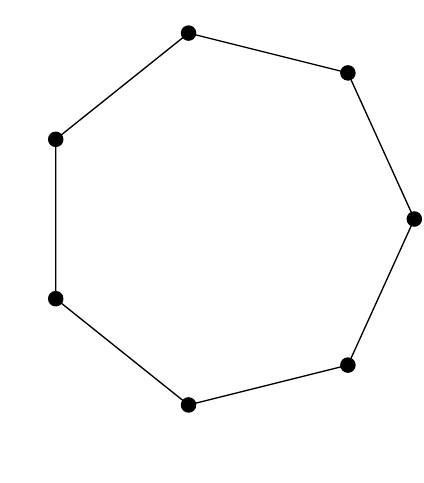}%
\end{picture}%
\setlength{\unitlength}{3355sp}%
\begingroup\makeatletter\ifx\SetFigFont\undefined%
\gdef\SetFigFont#1#2#3#4#5{%
  \reset@font\fontsize{#1}{#2pt}%
  \fontfamily{#3}\fontseries{#4}\fontshape{#5}%
  \selectfont}%
\fi\endgroup%
\begin{picture}(2388,2710)(1036,-2075)
\put(1426,239){\makebox(0,0)[lb]{\smash{{\SetFigFont{10}{12.0}{\rmdefault}{\mddefault}{\updefault}{\color[rgb]{0,0,0}$\Sigma^6$}%
}}}}
\put(1051,-661){\makebox(0,0)[lb]{\smash{{\SetFigFont{10}{12.0}{\rmdefault}{\mddefault}{\updefault}{\color[rgb]{0,0,0}$\Sigma^5$}%
}}}}
\put(2551,-1711){\makebox(0,0)[lb]{\smash{{\SetFigFont{10}{12.0}{\rmdefault}{\mddefault}{\updefault}{\color[rgb]{0,0,0}$\Sigma^3$}%
}}}}
\put(3301,-136){\makebox(0,0)[lb]{\smash{{\SetFigFont{10}{12.0}{\rmdefault}{\mddefault}{\updefault}{\color[rgb]{0,0,0}$\Sigma^1$}%
}}}}
\put(3301,-1111){\makebox(0,0)[lb]{\smash{{\SetFigFont{10}{12.0}{\rmdefault}{\mddefault}{\updefault}{\color[rgb]{0,0,0}$\Sigma^2$}%
}}}}
\put(2476,464){\makebox(0,0)[lb]{\smash{{\SetFigFont{10}{12.0}{\rmdefault}{\mddefault}{\updefault}{\color[rgb]{0,0,0}$\Sigma^7$}%
}}}}
\put(1426,-1486){\makebox(0,0)[lb]{\smash{{\SetFigFont{10}{12.0}{\rmdefault}{\mddefault}{\updefault}{\color[rgb]{0,0,0}$\Sigma^4$}%
}}}}
\end{picture}%
\caption{\label{fig:cyclic}}
\end{centering}
\end{figure}%

Suppose for concreteness that $p(y) = y^{r+1} - 1$, and consider straight paths from one root of unity to the next (in clockwise order), as in Figure \ref{fig:cyclic}. This gives spheres $\Sigma^1,\dots,\Sigma^{r+1} \subset Y$. For specific choices of orientations, and assuming $r>1$, these satisfy
\begin{equation} \label{eq:s-matrix}
\Sigma^i \cdot \Sigma^j = \begin{cases} 1-(-1)^n & i = j, \\
-1 & j=i-1, \\
(-1)^n & j = i+1, \\
0 & \text{otherwise.}
\end{cases}
\end{equation}
(There is a relation $[\Sigma^1] + \cdots + [\Sigma^{r+1}] = 0$ in homology, but we find it convenient to use all the $\Sigma^i$, for greater symmetry; one can also include the trivial case $r = 1$, in which the only relevant information is the selfintersection number of $[\Sigma^1] = -[\Sigma^2]$).

Let's see how Picard-Lefschetz theory recovers \eqref{eq:s-matrix}. The smooth fibre of the $y$-projection is an affine quadric, hence diffeomorphic to $T^*S^{n-2}$. All vanishing cycles of the projection are equal to the zero-section $S^{n-2} \subset T^*S^{n-2}$. If we choose a basis of such cycles, their intersection numbers form a matrix \eqref{eq:b-matrix} with $B_{ij} = \chi(S^{n-2}) = 1+(-1)^n$. The nullspace of $B$ is spanned by $s^1 = (-1,1,0,\dots,0)$ and its cyclic permutations $s^2,\dots,s^{r+1}$, again with a relation $s^1 + \cdots + s^{r+1} = 0$. Applying \eqref{eq:intersection-number} reproduces \eqref{eq:s-matrix}:
\begin{equation}
(s^i)^* \,A \, s^j = \Sigma^i \cdot \Sigma^j. 
\end{equation}
\end{example}

\begin{example} \label{th:ab-example}
Fix coprime integers $0<a<b$, and another integer $n \geq 2$. Consider
\begin{equation} \label{eq:fieseler}
X = X_{a,b} = \{ x^a y^b - 1 = z_1^2 + \cdots + z_{n-1}^2 \},
\end{equation}
where $(x,y,z_1,\dots,z_{n-1})$ are the coordinates in $\bC^{n+1}$. As a Lefschetz fibration, take the map
\begin{equation} \label{eq:pi-map}
\begin{aligned}
& \pi: X \longrightarrow \bC, \\
& \pi(x,y,z_1,\dots,z_{n-1}) = ax+by.
\end{aligned}
\end{equation}
The fibres of $\pi$ are the hypersurfaces encountered in Example \ref{th:an-milnor}, for $r = a+b-1$. Concretely, take $Y= \pi^{-1}(0)$, which is \eqref{eq:am} for $p(y) = y^{a+b}(-b/a)^a - 1$. An easy branch locus computation \cite[Section 3.3]{auroux-katzarkov-orlov04} shows that \eqref{eq:pi-map} has a basis of vanishing cycles $V^i \subset Y$, which are themselves fibered over paths in the $y$-plane as in Figure \ref{fig:ab-cycles}. In homology, this means that 
\begin{equation} \label{eq:s-relation}
[V^i] = [\Sigma^i] + [\Sigma^{i+1}] + \cdots + [\Sigma^{i+a-1}]. 
\end{equation}
From that and \eqref{eq:s-matrix}, one computes the matrix \eqref{eq:b-matrix}. It turns out to be a cyclic band matrix: if rows and columns are indexed by $\bZ/(a+b)$, the entries $B_{ij}$ depend only on $i-j$. Concretely,
\begin{equation} \label{eq:special-b-matrix}
B_{ij} = \begin{cases} 
(-1)^n & i-j = -a\; \mathrm{mod}\; (a+b), \\
1 + (-1)^n & i-j = -a+1,\dots,-1\; \mathrm{mod}\; (a+b), \\
1 - (-1)^n & i = j, \\
-1 - (-1)^n & i-j = 1,\dots,a-1\; \mathrm{mod}\; (a+b), \\
-1 & i-j = a\; \mathrm{mod}\; (a+b), \\
0 & \text{otherwise.}
\end{cases}
\end{equation}
For instance, if $(a,b) = (2,3)$ and $n$ is even, one has
\begin{equation}
B = \begin{pmatrix} 
0 & 2 & 1 & -1 & -2 \\
-2 & 0 & 2 & 1 & -1 \\
-1 & -2 & 0 & 2 & 1 \\
1 & -1 & -2 & 0 & 2 \\
2 & 1 & -1 & -2 & 0
\end{pmatrix},
\qquad
A = 
\begin{pmatrix}
1 & 2 & 1 & -1 & -2 \\
0 & 1 & 2 & 1 & -1 \\
0 & 0 & 1 & 2 & 1 \\
0 & 0 & 0 & 1 & 2 \\
0 & 0 & 0 & 0 & 1
\end{pmatrix}.
\end{equation}
\begin{figure}
\begin{centering}
\begin{picture}(0,0)%
\includegraphics{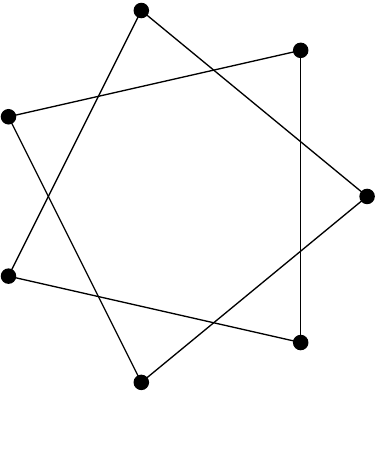}%
\end{picture}%
\setlength{\unitlength}{3355sp}%
\begingroup\makeatletter\ifx\SetFigFont\undefined%
\gdef\SetFigFont#1#2#3#4#5{%
  \reset@font\fontsize{#1}{#2pt}%
  \fontfamily{#3}\fontseries{#4}\fontshape{#5}%
  \selectfont}%
\fi\endgroup%
\begin{picture}(2121,2595)(1303,-2084)
\put(2951,-661){\makebox(0,0)[rb]{\smash{{\SetFigFont{10}{12.0}{\rmdefault}{\mddefault}{\updefault}{\color[rgb]{0,0,0}$V^1$}%
}}}}
\put(1651,-2011){\makebox(0,0)[lb]{\smash{{\SetFigFont{10}{12.0}{\rmdefault}{\mddefault}{\updefault}{\color[rgb]{0,0,0}$(a,b) = (2,5)$}%
}}}}
\put(2326,-186){\makebox(0,0)[rb]{\smash{{\SetFigFont{10}{12.0}{\rmdefault}{\mddefault}{\updefault}{\color[rgb]{0,0,0}$V^6$}%
}}}}
\put(2701,-261){\makebox(0,0)[rb]{\smash{{\SetFigFont{10}{12.0}{\rmdefault}{\mddefault}{\updefault}{\color[rgb]{0,0,0}$V^7$}%
}}}}
\put(2801,-1036){\makebox(0,0)[rb]{\smash{{\SetFigFont{10}{12.0}{\rmdefault}{\mddefault}{\updefault}{\color[rgb]{0,0,0}$V^2$}%
}}}}
\put(2301,-1131){\makebox(0,0)[rb]{\smash{{\SetFigFont{10}{12.0}{\rmdefault}{\mddefault}{\updefault}{\color[rgb]{0,0,0}$V^3$}%
}}}}
\put(2001,-511){\makebox(0,0)[rb]{\smash{{\SetFigFont{10}{12.0}{\rmdefault}{\mddefault}{\updefault}{\color[rgb]{0,0,0}$V^5$}%
}}}}
\put(2001,-861){\makebox(0,0)[rb]{\smash{{\SetFigFont{10}{12.0}{\rmdefault}{\mddefault}{\updefault}{\color[rgb]{0,0,0}$V^4$}%
}}}}
\end{picture}%
\caption{\label{fig:ab-cycles}}
\end{centering}
\end{figure}%

Suppose first that $n$ is odd. Then, the nullspace of \eqref{eq:special-b-matrix} is generated by $h = (1,\dots,1)$. This implies that the image of the map $H_n(X) \rightarrow H_\pi$ is of rank at most one (one can check that the image is in fact exactly $\bZ h$). One has
\begin{equation} \label{eq:null-hh}
h^* \, A \, h = 0.
\end{equation}
In the simplest case $a = 1$ and $n = 3$, one can show that all $X_{1,b}$ are diffeomorphic to the cotangent bundle $T^*S^3$, in a way which is compatible with the homotopy class of the almost complex structure (see \cite{maydanskiy-seidel09} for more information about the case $a = 1$ for general odd $n$).

The situation for even $n$ is a little more interesting. As before, $h$ lies in the nullspace of $B$. If $r$ is odd, it generates that nullspace. If $r$ is even, the nullspace has rank $2$, with the other generator being $\bar{h} = (1,0,1,0,\dots)$. However, going back to \eqref{eq:s-relation}, one sees that the kernel of the boundary map $H_\pi \rightarrow H_{n-1}(Y)$ is $\bZ h$ for any $r$. Hence, the intersection pairing on $H_n(X)$ is fully described by
\begin{equation} \label{eq:2ab}
h^* \, A \, h = 2ab.
\end{equation}
In the simplest instances, namely for $n = 2$ and $a = 1$, $X$ is a Danielewski surface \cite{dubouloz05}, diffeomorphic to the total space of the complex line bundle of degree $-2b$ over $S^2$. The case $n = 2$ and $a>1$ is more complicated: one gets a four-manifold with a circle action, and its boundary (at infinity) is a Seifert fibered space whose base has orbifold Euler characteristic $2/a+2/b-2 < 0$.

The vanishing of \eqref{eq:null-hh} is obvious from a topological viewpoint, since intersection numbers are skew-symmetric in odd dimensions. For the even-dimensional situation, the fact that \eqref{eq:2ab} is even again has a general topological explanation: since $X$ is a hypersurface in affine space, $TX$ is stably trivial, and such manifolds necessarily have even intersection form (this follows from the Wu formula, which expresses the mod $2$ selfintersection in $H_n(X)$ as the pairing with a characteristic class of $TX$ \cite{wu55,kervaire57}).
%
\end{example}

\begin{example} \label{th:mirrorp2-example}
As a variation of the previous situation (specialized to $a = 1$, $b = 2$), consider
\begin{equation} \label{eq:tilde-x}
X = \{ (x y^2 - 1)x = z_1^2 + \cdots + z_{n-1}^2\},
\end{equation}
which has
\begin{equation}
H_*(X) \iso \begin{cases} \bZ & \ast = 0,\, n-1,\, n, \\ 0 & \text{otherwise.}
\end{cases}
\end{equation}
To apply Picard-Lefschetz theory, we again use $\pi(x,y,z_1,\dots,z_{n-1}) = x+2y$. The fibre $Y = \pi^{-1}(0)$ is \eqref{eq:am} for the poynomial $p(y) = y^4/4 + y/2$, whose zeros are $\{0\} \cup \sqrt[3]{-2}$.
%
\begin{figure}
\begin{centering}
\begin{picture}(0,0)%
\includegraphics{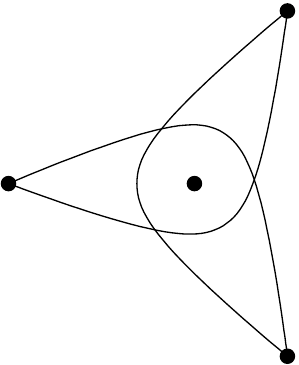}%
\end{picture}%
\setlength{\unitlength}{3355sp}%
\begingroup\makeatletter\ifx\SetFigFont\undefined%
\gdef\SetFigFont#1#2#3#4#5{%
  \reset@font\fontsize{#1}{#2pt}%
  \fontfamily{#3}\fontseries{#4}\fontshape{#5}%
  \selectfont}%
\fi\endgroup%
\begin{picture}(1671,2044)(1303,-1608)
\put(1501,-961){\makebox(0,0)[lb]{\smash{{\SetFigFont{10}{12.0}{\rmdefault}{\mddefault}{\updefault}{\color[rgb]{0,0,0}$V^3$}%
}}}}
\put(2301, 89){\makebox(0,0)[lb]{\smash{{\SetFigFont{10}{12.0}{\rmdefault}{\mddefault}{\updefault}{\color[rgb]{0,0,0}$V^1$}%
}}}}
\put(1701,-286){\makebox(0,0)[lb]{\smash{{\SetFigFont{10}{12.0}{\rmdefault}{\mddefault}{\updefault}{\color[rgb]{0,0,0}$V^2$}%
}}}}
\end{picture}%
\caption{\label{fig:mirrorp2}}
\end{centering}
\end{figure}%
A basis of vanishing cycles in $Y$ consists of the three spheres fibered over the paths drawn in Figure \ref{fig:mirrorp2}. 
The analogue of \eqref{eq:special-b-matrix} is
\begin{equation} \label{eq:mirrorp2-b}
B = \begin{pmatrix}
1-(-1)^n & -1-2(-1)^n & 2+(-1)^n \\ 
2+(-1)^n & 1-(-1)^n & -2-(-1)^n  \\ -1-2(-1)^n & 1+2(-1)^n & 1-(-1)^n
\end{pmatrix}.
\end{equation}
The nullspace of $B$ is generated by $h = ((-1)^n,1,1)$, and
\begin{equation} \label{eq:tilde-h}
h^* \, A \, h = 0,
\end{equation}
which implies that the intersection pairing on $H_n(X)$ is zero for any $n$. 

The case $n = 2$ is a well-known one in mirror symmetry. In that dimension, the subset of $X = \{x^2y^2 - x = z^2\}$ where $x \neq 0$ can be identified with $(\bC^*)^2$ by setting
\begin{equation}
u_1 = \frac{1}{xy-z}, \quad u_2 = \frac{1}{xy+z},
\end{equation}
and then
\begin{equation}
\pi(u_1,u_2) = u_1+u_2+\frac{1}{u_1u_2}
\end{equation}
is the Landau-Ginzburg superpotential mirror to $\bC P^2$. The fibres of this superpotential are three-punctured (smooth or nodal) curves, and $X$ is obtained from $(\bC^*)^2$ by filling in one of the punctures in each fibre. This explains why Figure \ref{fig:mirrorp2} reproduces \cite[Figure 5]{auroux-katzarkov-orlov04}. It also shows that $X$ contains a totally real torus (namely, $S^1 \times S^1 \subset (\bC^*)^2 \subset X$) which is nonzero in homology. This gives an alternative explanation for the vanishing of the intersection pairing, via \eqref{eq:chi}.

There is a similar topological viewpoint for general $n$. The subset of \eqref{eq:tilde-x} where $x \neq 0$ can be identified with $\bC^n \setminus \{z_1^2 + \cdots + z_n^2 = 0\}$, by setting 
\begin{equation} \label{eq:zn}
z_n = \sqrt{-1}\, xy. 
\end{equation}
This subset retracts to a totally real submanifold $L$, namely the image of the embedding
\begin{equation} \label{eq:s1sn}
\begin{aligned}
& S^1 \times_{\bZ/2} S^{n-1} \longrightarrow \bC^n \setminus \{z_1^2 + \cdots + z_n^2 = 0\}, \\
& (\zeta,v) \longmapsto \zeta v.
\end{aligned}
\end{equation}
Here, $\zeta \in S^1 \subset \bC$, $v \in S^{n-1} \subset \bR^n$, and $\bZ/2$ acts on both by the antipodal map. Up to homotopy equivalence, $X$ can be obtained from $S^1 \times_{\bZ/2} S^{n-1}$ by contracting the subset $S^1 \times_{\bZ/2} S^{n-2}$, so that $\tilde{H}_*(X;\bQ) \iso H^*(S^1;\bQ) \otimes H^*(S^{n-1},S^{n-2};\bQ)^{\bZ/2}$. If $n$ is even, \eqref{eq:s1sn} is orientable and represents a nontrivial homology class, which again allows one to use \eqref{eq:chi} to determine the intersection pairing.
\end{example}

\subsection{$\bC^*$-actions on Fukaya categories\label{subsec:dilating}}
At this point, we add symplectic structures to the discussion. This allows one to overcome the limitations of a purely topological theory (such as that encountered in Example \ref{th:ab-example} for odd $n$, where the intersection pairing is the same for all $a,b$). For technical reasons, we impose additional assumptions as follows:

\begin{setup} \label{th:setup-symplectic}
Throughout, we will work with exact symplectic manifolds with corners 
\begin{equation}
(M^{2n},\omega_M = d\theta_M,I_M)
\end{equation}
 in the sense of \cite[Section 7]{seidel04}. Additionally, any such manifold is assumed to come with a trivialization of the canonical bundle $K_M = \Lambda^n_{\bC}(TM^*)$ (only the homotopy class of that trivialization matters).

An exact Lagrangian brane in $M$ is a Lagrangian submanifold $L \subset $M which is closed, connected, exact (meaning that $\theta_M|L$ is an exact one-form), and comes with a choice of grading (which in particular determines an orientation) as well as a {\em Spin} structure. 
\end{setup}

The class of exact symplectic manifolds with corners includes Liouville domains, which are the special case where $M$ has smooth boundary, and near that boundary, $\theta_M = -dh \circ I_M$ for some function satisfying $h|\partial M = 0$, $dh>0$ in outwards direction (conversely, given any exact symplectic manifold with corners, one can shrink it slightly and modify the almost complex structure, so as to make it into a Liouville domain; hence, allowing manifolds with corners does not add substantially to the level of generality, but it is convenient when discussing Lefschetz fibrations).

\begin{example} \label{th:hypersurface}
Let $X = \{p(x_1,\dots,x_{n+1}) = 0\} \subset \bC^{n+1}$ be a smooth affine hypersurface. Choose a real polynomial $h = h(x_1,\bar{x}_1,\dots,x_{n+1},\bar{x}_{n+1}): \bC^{n+1} \rightarrow \bR$ which is proper, bounded below, and strictly plurisubharmonic. Take $R$ larger than the largest critical value of $h|X$. Then, $M = \{x \; : \; h(x) \leq R\}$ equipped with $\theta_M = -d^ch$, $\omega_M = d\theta_M$, and the standard complex structure $I_M = i$, is a Liouville domain. Moreover, up to deformation inside the class of such domains, it is independent of the choice of $h$ and $R$ \cite[Proposition 3.3]{caubel-tibar03}. One equips it with the trivialization of $K_M$ given by the complex volume form $\mathrm{res}(dx_1 \wedge \cdots \wedge dx_{n+1}/p)$.
\end{example}

To an $M$ as in Setup \ref{th:setup-symplectic}, one can associate its Fukaya $A_\infty$-category $\scrF(M)$ (this is the $\bZ$-graded version, with coefficients in $\bC$). The objects of $\scrF(M)$ are exact Lagrangian branes, and the cohomology level morphism spaces are Lagrangian Floer cohomology groups
\begin{equation} \label{eq:floer-cohomology}
H^*(\hom_{\scrF(M)}(L_0,L_1)) = \mathit{HF}^*(L_0,L_1) ,
\end{equation}
whose Euler characteristic equals the intersection number:
\begin{equation} \label{eq:floer-euler}
\chi(\mathit{HF}^*(L_0,L_1)) = L_0 \cdot L_1.
\end{equation}
For $L_0 = L_1 = L$, there is a canonical algebra isomorphism which refines the equality \eqref{eq:chi}:
\begin{equation} \label{eq:endo-hf}
H^*(\hom_{\scrF(M)}(L,L)) \iso H^*(L;\bC).
\end{equation}

The next topic is the main black box of our current discussion: we will suppose from now on that {\em $M$ has a dilating $\bC^*$-action}. What we actually mean by this is that the Fukaya category $\scrF(M)$ admits such an action, in a sense which will be made precise in Definitions \ref{th:fuk-has-an-action} and \ref{th:dilating-action}; and that we work with a fixed such action. However, for the moment the more compact, if slightly more vague, terminology may be permitted. In this situation, each exact Lagrangian brane $L \subset M$ comes with a distinguished deformation class
\begin{equation} \label{eq:l-def-class}
\mathit{Def}_L^0 \in \mathit{HF}^1(L,L) \iso H^1(L;\bC).
\end{equation}
If that class vanishes (in particular, if $H^1(L;\bC) = 0$), $L$ can be made equivariant, in a sense which will be specified later on; see Definition \ref{th:equivariant-l}. Making $L$ equivariant involves a choice, and the possible choices form an affine space over $\bZ$. We denote the change of equivariant structure by
\begin{equation} \label{eq:equi-shift}
L \longmapsto L \langle k \rangle, \;\; k \in \bZ.
\end{equation}
Note that this is different from the (downwards) shift in the grading, denoted by
\begin{equation}
L \longmapsto L[k].
\end{equation}

\begin{example} \label{th:cylinder}
Take the cylinder $M = [-1,1] \times S^1$ (in coordinates $z = p+iq$), with $\omega_M = dp \wedge dq$, $\theta_M = p \, \mathit{dq}$, and the trivialization of its canonical bundle given by $dz = dp+i \,dq$. Up to quasi-isomorphism and shifts in the grading, $\scrF(M)$ contains only two objects, given by the exact circle in $M$ with either of its two {\em Spin} structures. One can equip $M$ with a dilating $\bC^*$-action such that \eqref{eq:l-def-class} vanishes for both of these objects.
\end{example}

%

Take two exact Lagrangian branes $(L_0,L_1)$ which have been made equivariant. Their Floer cohomology \eqref{eq:floer-cohomology} then becomes a graded representation of $\bC^*$. In particular, for $L_0 = L_1 = L$ one gets a $\bC^*$-action on the ordinary cohomology \eqref{eq:endo-hf}. Here are some properties:
\begin{itemize} \itemsep0.5em
\item The $\bC^*$-action on $\mathit{HF}^*(L_0\langle k_0 \rangle, L_1\langle k_1 \rangle)$ is obtained from that on $\mathit{HF}^*(L_0,L_1)$ by tensoring with the one-dimensional representation of weight $k = k_1-k_0$. We denote this tensor product operation by $\langle k \rangle$ as well, so that the statement can be written as
\begin{equation} \label{eq:shift-equivariance}
\mathit{HF}^*(L_0\langle k_0 \rangle, L_1\langle k_1 \rangle) \iso
\mathit{HF}^*(L_0,L_1) \langle k_1 - k_0 \rangle.
\end{equation}
\item If $(L_0,L_1,L_2)$ are equivariant, the composition (triangle product)
\begin{equation}
\mathit{HF}^*(L_1,L_2) \otimes \mathit{HF}^*(L_0,L_1) \longrightarrow \mathit{HF}^*(L_0,L_2)
\end{equation}
is $\bC^*$-equivariant.
\item If $L$ is equivariant, the $\bC^*$-action on \eqref{eq:endo-hf} has weight $1$ in degree $n$.
\end{itemize}
The second property implies that for an equivariant $L$, the $\bC^*$-action on $H^*(L;\bC)$ is by algebra automorphisms, hence acts trivially on the identity (in degree $0$). The last property is what makes the action a dilating one. We want to draw one immediate consequence of those two properties. A basic feature of Floer cohomology is that the pairings
\begin{equation} \label{eq:two-sided-product}
\mathit{HF}^{n-*}(L_1,L_0) \otimes \mathit{HF}^*(L_0,L_1) \xrightarrow{\text{product}} \mathit{HF}^n(L_0,L_0) \iso H^n(L_0;\bC) \xrightarrow{\int_{L_0}} \bC
\end{equation}
are nondegenerate, and symmetric in the sense that
\begin{equation}
\textstyle \int_{L_0} x_1 x_0 = (-1)^{|x_1|\,|x_0|} \int_{L_1} x_0 x_1.
\end{equation}
It follows that as $\bC^*$-representations, we have the following twisted version of duality:
\begin{equation} \label{eq:q-symmetry}
\mathit{HF}^*(L_1,L_0) \iso \mathit{HF}^{n-*}(L_0,L_1)^\vee \otimes \mathit{HF}^n(L_k,L_k) \iso \mathit{HF}^{n-*}(L_0,L_1)^\vee \langle 1 \rangle.
\end{equation}

In analogy with \cite{seidel-solomon10}, one defines the $q$-intersection number to be 
\begin{equation} \label{eq:q-intersection}
L_0 \cdot_q L_1 = \mathrm{Str}\big(q: \mathit{HF}^*(L_0,L_1) \longrightarrow \mathit{HF}^*(L_0,L_1)\big) \in \bZ[q,q^{-1}].
\end{equation}
Here, the endomorphism of $\mathit{HF}^*(L_0,L_1)$ is the action of $q \in \bC^*$, and $\mathrm{Str}$ is the supertrace (Lefschetz trace). Equivalently, if $\mathit{HF}^*(L_0,L_1)^k$ is the subspace on which $\bC^*$ acts with weight $k$, the coefficients of \eqref{eq:q-intersection} are the Euler characteristics
\begin{equation} \label{eq:q-intersection-2}
L_0 \cdot_q L_1 = \sum_k\chi(\mathit{HF}^*(L_0,L_1)^k)\, q^k.
\end{equation}
These $q$-intersection numbers satisfy the following:
\begin{itemize} \itemsep0.5em 

\item Specializing to $q = 1$ recovers the intersection number $L_0 \cdot L_1$.

\item If $L_0 \cdot_q L_1 = \sum_k a_k q^k$, then $\sum_k |a_k|$ is a lower bound for the total dimension of $\mathit{HF}^*(L_0,L_1)$. More precisely, $\sum_k \mathrm{max}(0,a_k)$ is a lower bound for the total dimension of $\mathit{HF}^{\mathit{even}}(L_0,L_1)$, and $-\sum_k \mathrm{min}(0,a_k)$ correspondingly for $\mathit{HF}^{\mathit{odd}}(L_0,L_1)$.

\item For any equivariant $L$, 
\begin{equation} \label{eq:self-bullet}
L \cdot_q L = 1 + (-1)^n q + \sum_{j \in J} (-1)^{|x_j|} q^{w(x_j)},
\end{equation}
where the sum is over a basis $\{x_j\}_{j \in J}$ of $\bigoplus_{0 < \ast < n} H^*(L;\bC)$. The $|x_j|$ are the degrees of the basis elements, and the $w(x_j)$ are a priori unknown integers. In the simplest case where $L$ is a rational homology sphere, the sum is empty and one gets
\begin{equation} \label{eq:self-bullet-2}
L \cdot_q L = 1 + (-1)^n q.
\end{equation}

\item Changing the equivariant structure affects the $q$-intersection number as follows:
\begin{equation} \label{eq:q-rescale}
L_0 \langle k_0 \rangle \cdot_q L_1 \langle k_1 \rangle = q^{k_1-k_0} (L_0 \cdot_q L_1).
\end{equation}

\item Shifting the grading affects the $q$-intersection number as follows:
\begin{equation} \label{eq:q-shift}
L_0[k_0] \cdot_q L_1[k_1] = (-1)^{k_1-k_0} (L_0 \cdot_q L_1).
\end{equation}
\item For any equivariant $(L_0,L_1)$,
\begin{equation} \label{eq:dual-bullet}
L_1 \cdot_q L_0 = (-1)^n\, q\, (L_0 \cdot_q L_1)^*,
\end{equation}
where the superscript $*$ stands for the substitution $q \mapsto q^{-1}$. 
\item Let $L_0,L_1,L_2$ be equivariant. Assume that $L_0$ is a sphere, and let $\tau_{L_0}$ be the associated Dehn twist. Then there is a way of making $\tau_{L_0}(L_1)$ equivariant, such that
\begin{equation} \label{eq:q-picard-lefschetz}
L_2 \cdot_q \tau_{L_0}(L_1) = L_2 \cdot_q L_1 - (L_2 \cdot_q L_0)(L_0 \cdot_q L_1).
\end{equation}
Similarly, for an inverse Dehn twist
\begin{equation} \label{eq:q-picard-lefschetz-inverse}
L_2 \cdot_q \tau_{L_0}^{-1}(L_1) = L_2 \cdot_q L_1 - (-1)^n q^{-1} (L_2 \cdot_q L_0)(L_0 \cdot_q L_1).
\end{equation}
\end{itemize}
Except for \eqref{eq:q-picard-lefschetz} and its counterpart \eqref{eq:q-picard-lefschetz-inverse} (to which we will return later), all these properties are consequences of the previously listed ones for the $\bC^*$-action. For instance, \eqref{eq:dual-bullet} follows from \eqref{eq:q-symmetry} by taking Lefschetz traces. To make \eqref{eq:q-shift} precise, we have to specify how the shifted brane $L[k]$ is made equivariant: this is done in such a way that the $\bC^*$-action on $\mathit{HF}^{-k}(L,L[k]) \iso \mathit{HF}^0(L,L) \iso H^0(L;\bC)$ is trivial. Given that, the compatibility of the $\bC^*$-action with the product implies that the $\bC^*$-action on $\mathit{HF}^*(L_0[k_0],L_1[k_1])$ is preserved by the isomorphism between that group and $\mathit{HF}^{*+k_1-k_0}(L_0,L_1)$; the desired formula for $q$-intersection numbers follows from that.

\begin{remark}
There is a more precise version of \eqref{eq:self-bullet}, obtained by combining the contributions of Poincar{\'e} dual generators. This has the form
\begin{equation}
L \cdot_q L = (1 + (-1)^n q) + \sum_{j \in J'} (-1)^{|x_j|} (q^{w(x_j)} + (-1)^n q^{1-w(x_j)}).
\end{equation}
If $n$ is odd, the sum is over a basis $\{x_j\}_{j \in J'}$ of the subspace $\bigoplus_{0 < \ast < n/2} H^*(L;\bC)$. If $n$ is even, one has to include a half-dimensional subspace of $H^{n/2}(L;\bC)$ as well (it follows from the existence of the $\bC^*$-action that $H^{n/2}(L;\bC)$ is even-dimensional for any equivariant $L$, which is nontrivial if $n$ is a multiple of $4$).
\end{remark}

\subsection{$\bC^*$-actions in Picard-Lefschetz theory}
In the symplectic context, Lefschetz fibrations become a particularly powerful tool.

\begin{setup} \label{th:symplectic-lefschetz}
We will work with Lefschetz fibrations which topologically look like those in Setup \ref{th:topological-lefschetz}, but now assuming that they carry the structure of exact symplectic Lefschetz fibrations in the sense of \cite[Section 17]{seidel04}. In particular, both the total space $E$ and fibre $F$ are exact symplectic manifolds with corners. We will also require that $E$ comes with a trivialization of $K_E$, which then induces the corresponding structure for $F$ (in a way which is unique up to homotopy).

Lefschetz thimbles and vanishing cycles are exact Lagrangian submanifolds. Because it bounds a ball, each vanishing cycle can be equipped with a grading (non-uniquely) and a {\em Spin} structure (uniquely determined by having to bound one on a ball). In this way, the vanishing cycle becomes an object of the Fukaya category $\scrF(F)$.
\end{setup}

As a consequence of applying Picard-Lefschetz theory in our context, one gets the following (which is one of the two main theorems in this paper):

\begin{theorem} \label{th:gm-picard-lefschetz}
Suppose that $F$ admits a dilating $\bC^*$-action, and that there is a basis of vanishing cycles $(V^1,\dots,V^m)$ such that $\mathit{Def}^0_{V^k} = 0$ for all $k$ (the last-mentioned condition is automatically true if $2n =\mathrm{dim}(E) \geq 6$, since then $H^1(V^k;\bC) = 0$). Then $E$ also admits a dilating $\bC^*$-action.
\end{theorem}

\begin{example} \label{th:simple-milnor}
Consider \eqref{eq:am}, for $n \geq 3$ (which means, excluding the case of hyperelliptic curves) and any $r$. For $n = 3$, this is the total space of a Lefschetz fibration whose fibre is a cylinder, hence carries a dilating $\bC^*$-action by Theorem \ref{th:gm-picard-lefschetz} and Example \ref{th:cylinder}. For $n > 3$ the fibre is $T^*S^{n-2}$, whose Fukaya category is completely understood, and easily seen to admit a dilating $\bC^*$-action as well; alternatively, one can argue by induction on the dimension (since cotangent bundles of spheres also appear as the special case $r = 1$). This is in fact the example considered in \cite{seidel12}.
\end{example}

\begin{example} \label{th:stabilized}
Let $p \in \bC[x_1,\dots,x_{n+1}]$ be a polynomial with an isolated critical point at the origin. Suppose that the Hessian of $p$ at that point has rank $\geq 2$. Then the Milnor fibre of the singularity admits a dilating $\bC^*$-action. This follows from Theorem \ref{th:gm-picard-lefschetz}, Example \ref{th:simple-milnor}, and an iterated Lefschetz fibration argument, which is the same as in \cite[Example 2.13]{seidel13}.
\end{example}

We now turn to more concrete implications. For the rest of this discussion, let's fix an exact symplectic Lefschetz fibration satisfying the conditions of Theorem \ref{th:gm-picard-lefschetz}, with its basis of vanishing cycles $(V^1,\dots,V^m)$. We make the $V^k$ into equivariant objects of $\scrF(F)$. Form a matrix $B_q$ from their $q$-intersection numbers \eqref{eq:q-intersection}:
\begin{equation} \label{eq:q-b-matrix} 
B_{q,ij} = V^i \cdot_q V^j \in \bZ[q,q^{-1}].
\end{equation}
Specializing to $q = 1$ then recovers \eqref{eq:b-matrix}. Similarly, we can introduce an upper triangular matrix $A_q$ which is a $q$-deformation of \eqref{eq:a-matrix}, namely
\begin{equation} \label{eq:q-a-matrix}
A_{q,ij} = \begin{cases} 
B_{q,ij} & i<j, \\ 
1 & i=j, \\
0 & \text{otherwise.}
\end{cases}
\end{equation}
$B_q$ determines $A_q$. Conversely, from \eqref{eq:self-bullet-2} and \eqref{eq:dual-bullet} we get a $q$-deformed version of \eqref{eq:ab-relation}:
\begin{equation} \label{eq:aqbq-relation}
B_q = A_q - q (-1)^n A_q^*,
\end{equation}
where $*$ stands for transposition combined with the change of variables $q \mapsto q^{-1}$. We equip
\begin{equation} \label{eq:q-h}
H_{\pi,q} = \bZ[q,q^{-1}]^m
\end{equation}
with the nondegenerate bilinear pairing
\begin{equation} \label{eq:q-pi}
h_0 \cdot_{\pi,q} h_1 = h_0^* \, A_q \, h_1.
\end{equation}
This is hermitian, meaning that 
\begin{equation}
(q^{k_0}h_0) \cdot_{\pi,q} (q^{k_1}h_1) = q^{k_1-k_0} (h_0 \cdot_{\pi,q} h_1).
\end{equation}

\begin{proposition} \label{th:q-hurwitz}
Up to $\bZ[q,q^{-1}]$-linear isomorphism, the pair $(H_{\pi,q},\cdot_{\pi,q})$ is independent of the choice of basis of vanishing cycles.
\end{proposition}

\begin{proof}
Even in terms of our fixed basis $(V^1,\dots,V^m)$, we have made two auxiliary choices, namely those of grading and of the equivariant structure. In view of \eqref{eq:q-rescale}, passing from $V^i$ to $V^i\langle k_i \rangle$ amounts to multiplying the corresponding basis vectors of $H_{\pi,q}$ by $q^{k_i}$. Similarly, by \eqref{eq:q-shift}, changing the grading of $V^i$ by $k_i$ amounts to multiplying the corresponding basis vector by $(-1)^{k_i}$.

Beyond that, any two bases of vanishing cycles can be transformed into each other by a sequence of Hurwitz moves (and their inverses):
\begin{equation}
\tilde{V}^i = \begin{cases} V^i & i < k \text{ or } i>k+1, \\
\tau_{V^k}(V^{k+1}) & i = k, \\ 
V^k & i = k+1, \\
\end{cases}
\end{equation}
Using the $q$-Picard-Lefschetz formula \eqref{eq:q-picard-lefschetz}, one shows that the two associated matrices $A_q$ and $\tilde{A}_q$ are related by
\begin{equation} \label{eq:q-transition}
\tilde{A}_q = C_q^* A_q C_q,
\end{equation}
where $C_q$ is the invertible matrix (with Kronecker symbol notation)
\begin{equation}
C_{q,ij} = \begin{cases} \delta_{ij} & j \notin \{k,k+1\}, \\
\delta_{ik} & j = k+1, \\
0 & j = k \text{ and } i \notin \{k,k+1\}, \\
1 & j = k \text{ and } i = k+1, \\
-V^k \cdot_q V^{k+1}  & j = k \text{ and } i = k.
\end{cases}
\end{equation}
\end{proof}

Proposition \ref{th:q-hurwitz} says that \eqref{eq:q-pi} is an invariant of our exact symplectic Lefschetz fibration together with the choice of dilating $\bC^*$-action on the fibre (later on, we will give a slightly more conceptual explanation for this, in terms of the equivariant Grothendieck group of a suitable category). In parallel with \eqref{eq:characterize-m}, one can define the $q$-monodromy matrix $N_q \in \mathit{GL}_m(\bZ[q,q^{-1}])$ by asking that 
\begin{equation}
h_0 \cdot_{\pi,q} N_q h_1 = (-1)^n q (h_1 \cdot_{\pi,q} h_0)^*.
\end{equation}
This is clearly independent of the choice of basis. The analogue of \eqref{eq:monodromy} is
\begin{equation} \label{eq:q-monodromy}
N_q = (-1)^n q A_q^{-1} A_q^*.
\end{equation}
It is an interesting question to what extent $N_q$ is related to the geometric concept of monodromy. A partial answer is given by the following expression, which is obtained by repeatedly applying \eqref{eq:q-picard-lefschetz}. If $v^i$ are the standard basis vectors in $H_\pi$, then
\begin{equation}
(N_q v^i) \cdot_{\pi,q} v^j = (-1)^n q^{-1} ( A_q (A_q^{-1})^* A_q)_{ij} = \begin{cases}
V^i \cdot_q V^j + \mu(V^i) \cdot_q V^j & i<j, \\
1 + \mu(V^i) \cdot_q V^i & i = j, \\ 
\mu(V^i) \cdot_q V^j & i>j,
\end{cases} 
\end{equation}
where $\mu = \tau_{V^1} \cdots \tau_{V^m}$ is the monodromy acting on $F$. In particular, one can use this formula, and its analogues for powers of $N_q$ (involving iterates of $\mu$) to obtain lower bounds on the dimensions of the groups $\mathit{HF}^*(\mu^k(V^i),V^k)$. In analogy with the classical Lefschetz fixed point formula, it seems plausible to think that the trace of $N_q$, and of its iterates, should have an interpretation in terms of a suitable version of fixed point Floer cohomology on the total space $E$ (concretely, as the supertrace of a $\bC^*$-action on that Floer cohomology). We cannot resolve this issue here, but \cite[Conjecture 6.1]{seidel00b} may provide one possible approach.

\begin{remark}
If $A_q = A$ is constant in $q$, \eqref{eq:aqbq-relation} agrees with the previously considered \eqref{eq:givental}. It is possible that, for Lefschetz fibrations arising from the Morsification of isolated hypersurface singularities, there is a relation between the theory developed here and that in \cite{givental88}. However, in general the two theories are not equivalent (as shown in Example \ref{th:q-mirrorp2}, where the powers of $q$ appearing in a single matrix entry differ by more than $1$).
\end{remark}

The applications to the topology of Lagrangian submanifolds rely on the following $q$-analogue of \eqref{eq:h-condition-2}, \eqref{eq:intersection-number}.

\begin{theorem} \label{th:main}
In the situation of Theorem \ref{th:gm-picard-lefschetz}, any equivariant exact Lagrangian brane $L \subset E$ determines a class in \eqref{eq:q-h}, denoted by $l_q \in H_{\pi,q}$, whose reduction to $q = 1$ recovers $l = [L] \in \bZ^m \iso H_\pi$. All such classes satisfy
\begin{equation} \label{eq:bq-null}
B_q\, l_q = 0.
\end{equation}
Moreover, for any pair $(L_0,L_1)$ and the associated classes $(l_{0,q},l_{1,q})$, one has
\begin{equation}
l_{0,q} \cdot_{\pi,q} l_{1,q} = l_{0,q}^* \, A_q \, l_{1,q} = L_0 \cdot_q L_1. 
\label{eq:obtain-bullet}
\end{equation}
\end{theorem}

%

\begin{corollary} \label{th:primitive}
In the situation of Theorem \ref{th:main}, take a Lagrangian submanifold $L \subset E$ which is a rational homology sphere and {\em Spin}. Then the class $[L] \in H_n(E)/(\mathit{torsion})$ is nonzero and primitive. Moreover, if $L_1,\dots,L_r$ are such submanifolds which are pairwise disjoint (or disjoinable by Lagrangian isotopies), then the $[L_i]$ are linearly independent.
\end{corollary}

\begin{proof}
This follows the model of \cite{seidel12, seidel13}. If $n$ is even, \eqref{eq:chi} says that $[L] \cdot [L] = 2$, which implies the desired result. We will therefore assume from now on that $n$ is odd. In view of \eqref{eq:self-bullet-2}, \eqref{eq:obtain-bullet} simplifies to
\begin{equation} \label{eq:self-q}
l_q^* \, A_q \, l_q = 1-q.
\end{equation}
Suppose that $[L]$ is torsion, which means that the $q = 1$ specialization of $l_q$ vanishes. Then $l_q$ is a multiple of $(1-q)$, hence $l_q^* \, A_q l_q$ is a multiple of $(1-q^{-1})(1-q) = -q^{-1}(1-q)^2$, contradicting \eqref{eq:self-q}. Similarly, if $[L]$ is not primitive, $l_q = (1-q)x + py$ for some prime $p$, and then $l_q^* \, A_q l_q$ is a multiple of $(1-q)^2$ in $(\bZ/p)[q,q^{-1}]$, which is again a contradiction (take the derivative at $q = 1$).

In the situation of the last statement, we know that $L_i \cdot_q L_j = 0$ for $i \neq j$, since the Floer cohomology vanishes. Given a relation $\sum_k c_k [L_k] = 0$ with integer coefficients $c_k$, we have $\sum_k c_k l_{k,q} = (1-q)x$ for some $x$, hence $\sum_k c_k^2 (1-q)$ is a multiple of $(1-q)^2$, which implies that all the $c_k$ must vanish.
\end{proof}

For instance, Corollary \ref{th:primitive} applies to the Milnor fibres from Example \ref{th:stabilized}, which yields an improved version of \cite[Theorem 1.6]{seidel13}. 

\subsection{Examples reconsidered}
Besides general statements such as Corollary \ref{th:primitive}, Theorem \ref{th:main} can yield additional information when applied to specific manifolds (like the classical Picard-Lefschetz theory on which it is modelled). We illustrate this by returning to two examples considered previously. The computation of the relevant matrices \eqref{eq:q-b-matrix} is deferred to Section \ref{subsec:the-examples}; here, we only state the result of those computations, and discuss the implications.

\begin{example} \label{th:q-ab-example}
Consider the hypersurface $X$ from Example \ref{th:ab-example}. Assume from now on that $n \geq 3$. Then, the fibre $Y$ of $\pi: X \rightarrow \bC$ admits a dilating $\bC^*$-action by Example \ref{th:simple-milnor}, which allows us to apply Theorems \ref{th:gm-picard-lefschetz} and \ref{th:main}. The analogue of \eqref{eq:s-matrix} for $q$-intersection numbers in $Y$ turns out to be
\begin{equation}
\label{eq:deformed-s-matrix}
\Sigma^i \cdot_q \Sigma^j = \begin{cases} 1+q(-1)^{n-1} & i = j, \\
-1 & j=i-1, \\
q(-1)^n & j=i+1, \\
0 & \text{otherwise.}
\end{cases}
\end{equation}
The analogue of \eqref{eq:special-b-matrix} is again a cyclic band matrix, obtained by replacing $(-1)^n$ with $q(-1)^n$ everywhere:
\begin{equation} \label{eq:boldb}
B_{q,ij} = \begin{cases} 
q(-1)^n & i-j = -a\; \mathrm{mod}\; (a+b), \\
1 + q(-1)^n & i-j = -a+1,\dots,-1\; \mathrm{mod}\; (a+b), \\
1 - q(-1)^n & i = j, \\
-1 - q(-1)^n & i-j = 1,\dots,a-1\; \mathrm{mod}\; (a+b), \\
-1 & i-j = a\; \mathrm{mod}\; (a+b), \\
0 & \text{otherwise.}
\end{cases}
\end{equation}
The nullspace of $B_q$ has rank $1$, and is generated by $h = (1,\dots,1)$ (this is easily seen by power series expansion around $q = -(-1)^n$). Taking the triangular matrix $A_q$ obtained from $B_q$ as in \eqref{eq:q-a-matrix}, one finds that for any $f \in \bZ[q,q^{-1}]$,
\begin{equation} \label{eq:lambda-lambda}
(fh)^* \, A_q \, fh = f^*f \, ba(1+q(-1)^n).
\end{equation}
Since $ba \geq 2$ by assumption, \eqref{eq:lambda-lambda} can never equal $1+q(-1)^n$. In view of \eqref{eq:obtain-bullet} and \eqref{eq:self-bullet-2}, we have shown:
\begin{equation} \label{eq:no-spheres}
\text{\it $X$ cannot contain Lagrangian $\bQ$-homology spheres which are {\em Spin}.}
\end{equation}
In a slightly different direction, assume that $f$ is not zero. Then, neither is $f^* f$, which means that the lowest and highest power of $q$ in \eqref{eq:lambda-lambda} both have coefficients which are nonzero multiples of $ab$. By looking at \eqref{eq:self-bullet}, one concludes:
\begin{equation} \label{eq:total-betti}
\parbox{37em}{\it Suppose that $L \subset X$ is a closed Lagrangian submanifold which satisfies $H^1(L) = 0$, is {\em Spin}, and has nonzero homology class. Then, the total sum of the Betti numbers of $L$ is at least $2ab$.
}
\end{equation}

It is worth while comparing this to what's already known about $X$. For $n$ even, the existence of Lagrangian homology spheres can be ruled out by topological arguments based on \eqref{eq:2ab}, which of course is the $q = 1$ specialization of \eqref{eq:lambda-lambda}, and the same applies to \eqref{eq:total-betti}. As mentioned before, the case $a = 1$ (and any $b$ and $n$) belongs to the class of manifolds studied in \cite{maydanskiy-seidel09}; the results obtained there imply the following:
\begin{equation} \label{eq:empty}
\parbox{37em}{\it $X_{1,b}$ can't contain any closed exact Lagrangian submanifolds.}
\end{equation}
This is much stronger than \eqref{eq:no-spheres} or \eqref{eq:total-betti}, but it is unclear whether one should expect a statement like \eqref{eq:empty} to hold for general $(a,b)$.
%
\end{example}

\begin{example} \label{th:q-mirrorp2}
We return to Example \ref{th:mirrorp2-example}, again assuming $n \geq 3$. The counterpart of \eqref{eq:mirrorp2-b} for $q$-intersection numbers is
\begin{equation} \label{eq:boldb-2}
B_q = \begin{pmatrix} 
1 - q(-1)^n & - q^{-1}(-1)^n - 1 - (-1)^n q & 1 + q(-1)^n + q^2 \\
1 +q(-1)^n + q^2 & 1-q(-1)^n & -1-q(-1)^n - q^2 \\
-q^{-1}(-1)^n -1 - (-1)^n q & 1 + q^{-1} (-1)^n + q (-1)^n & 1 - q(-1)^n
\end{pmatrix}
\end{equation}
which has
\begin{equation} \label{eq:determinant}
\det(B_q) = 
q^{-2}(q-1)^2 (q+1)^2 (q-(-1)^n) (q^2+1) \neq 0.
\end{equation}
As in the previous example, this implies \eqref{eq:no-spheres}. In analogy with \eqref{eq:total-betti}, one also has:
\begin{equation} 
\parbox{37em}{Any closed Lagrangian submanifold in $X$ which has zero first Betti number and is {\em Spin} must be nullhomologous.}
\end{equation}
It is not known whether $X$ contains an exact Lagrangian submanifold, but for even $n$ we can show a weaker statement, namely that it contains an orientable (and {\em Spin}) Lagrangian submanifold with nonzero Floer cohomology. Namely, consider the submanifold $L$ from \eqref{eq:s1sn}. This is not Lagrangian with respect to the given symplectic form on $X$, which is the constant form in the variables $(x,y,z_1,\dots,z_{n-1})$, denoted here by $\omega_{(x,y,z_1,\dots,z_{n-1})}$. However, $L$ is Lagrangian for the degenerate form $\omega_{(z_1,\dots,z_n)}$, where $z_n$ is as in \eqref{eq:zn}; and by a Moser argument, one can perturb it so that it becomes Lagrangian for
\begin{equation} \label{eq:epsilon-form}
\omega_\epsilon = \omega_{(z_1,\dots,z_n)} + \epsilon\, \omega_{(x,y,z_1,\dots,z_{n-1})}, \quad \text{$\epsilon>0$ small.}
\end{equation}
This form yields a symplectic structure isomorphic to $\omega_{(x,y,z_1,\dots,z_{n-1})}$ (see the uniqueness statement made in Example \ref{th:hypersurface}), hence we may use it to define $\scrF(X)$. Denote the perturbed Lagrangian submanifold by $L_\epsilon$. The meromorphic volume form $dz_1 \wedge \cdots \wedge dz_n/(z_1^2 + \cdots + z_n^2)$ on $\bC^n$ extends to a smooth holomorphic volume form on $X$. For $n = 2$, $L_\epsilon$ has zero Maslov number, which ensures that $\mathit{HF}^*(L_\epsilon,L_\epsilon) \neq 0$. For $n>2$, $L_\epsilon$ is monotone (with minimal Maslov number $n-2$); since it represents a nonzero homology class, the Floer cohomology is again nonzero by \cite[Corollary 3.1]{albers10} (alternatively, for $n \geq 6$, one can obtain the same conclusion by applying the spectral sequence from \cite{oh96}).
\end{example}
%

\section{Homological algebra background\label{sec:background}}

\subsection{$A_\infty$-categories}
Throughout, we will work with small $A_\infty$-categories $\scrA$ defined over $\bC$. Our sign conventions follow \cite{seidel04}. This means that the $A_\infty$-associativity equations are
\begin{equation} \label{eq:ainfty}
\sum_{i,j} (-1)^\ast \mu_{\scrA}^{d-j+1}(a_d,\dots,a_{i+j+1},
\mu_{\scrA}^j(a_{i+j},\dots,a_{i+1}),a_i,\dots,a_1) = 0,
\end{equation}
with $\ast = |a_1|+\cdots+|a_i|-i$. In order to transition to the traditional conventions for chain complexes and chain maps, one should define the differential on $\hom_\scrA(X_0,X_1)$ to be
\begin{equation} \label{eq:signed-differential}
a \mapsto (-1)^{|a|} \mu^1_{\scrA}(a),
\end{equation}
and the composition as the chain map
\begin{equation} \label{eq:signed-multiplication}
a_2 \cdot a_1 = (-1)^{|a_1|} \mu^2_{\scrA}(a_2, a_1).
\end{equation}
In this section and the following one, all $A_\infty$-categories are assumed to be strictly unital, with strict identity endomorphisms denoted by $e_X \in \hom_{\scrA}^0(X,X)$. These satisfy a series of equations starting with
\begin{equation}
\begin{aligned}
& \mu^1_{\scrA}(e_X) = 0, \\
& \mu^2_{\scrA}(a,e_{X_0}) = a, \;\; \mu^2_{\scrA}(e_{X_1},a) = (-1)^{|a|} a \quad \text{for }a \in \hom_{\scrA}(X_0,X_1), \\
& \dots
\end{aligned} 
\end{equation}
The cohomology level map induced by \eqref{eq:signed-multiplication} makes $H^*(\scrA)$ into a $\bC$-linear graded category (in the classical sense), with identity morphisms $[e_X]$. Recall that $\scrA$ is called proper if the spaces $H^*(\hom_{\scrA}(X_0,X_1))$ are finite-dimensional (since these are graded vector spaces, we should clarify that this means of finite total dimension).

\subsection{Twisted complexes, $A_\infty$-modules}
Given any $\scrA$, one can introduce the larger $A_\infty$-category of twisted complexes $\Tw(\scrA)$. As an intermediate step, one considers the additive enlargement $\scrA^\oplus$, whose objects are formal direct sums 
\begin{equation} \label{eq:formal-sum}
C = \bigoplus_{f \in F} W_f \otimes X_f,
\end{equation}
where $F$ is some finite set, the $W_f$ are finite-dimensional graded vector spaces, and the $X_f$ are objects of $\scrA$. Morphisms are defined by
\begin{equation} \label{eq:formal-sum-morphism}
\hom_{\scrA^\oplus}(C_0,C_1) = \!\!\bigoplus_{f_0 \in F_0,\, f_1 \in F_1}\!\!\! \mathit{Hom}(W_{0,f_0},W_{1,f_1}) \otimes \hom_{\scrA}(X_{0,f_0},X_{1,f_1}),
\end{equation}
and the $A_\infty$-structure combines that of $\scrA$ with composition of linear maps between the $W_f$-spaces (and some auxiliary signs). A twisted complex is a pair $(C,\delta_C)$, where $C$ is as in \eqref{eq:formal-sum}, and the differential $\delta_C \in \hom^1_{\scrA^\oplus}(C,C)$ has the following two properties: it is strictly decreasing with respect to some filtration of $F$; and it satisfies the generalized Maurer-Cartan equation
\begin{equation} \label{eq:maurer-cartan}
\mu^1_{\scrA^\oplus}(\delta_C) + \mu^2_{\scrA^\oplus}(\delta_C,\delta_C) + \cdots = 0.
\end{equation}
The $A_\infty$-category $\Tw(\scrA)$ has twisted complexes as objects; the morphism spaces are the same as in $\scrA^\oplus$; and the $A_\infty$-structure is obtained from that of $\scrA^\oplus$ by a deformation which inserts the differentials arbitrarily many times. For instance, for $a \in \hom_{\Tw(\scrA)}(C_0,C_1)$, one sets
\begin{equation} \label{eq:rho-mu}
\mu^1_{\Tw(\scrA)}(a) = \mu^1_{\scrA^{\oplus}}(a)
+ \mu^2_{\scrA^\oplus}(\delta_{C_1},a) + \mu^2_{\scrA^{\oplus}}(a,\delta_{C_0}) + \cdots.
\end{equation}

If $C$ is a twisted complex and $D$ is a finite-dimensional chain complex of vector spaces, one can form a new twisted complex $D \otimes C$, by tensoring the spaces $W_f$ in \eqref{eq:formal-sum} with $D$, and equipping the outcome with a combined differential. Any  $D$ admits a filtration with respect to which the differential is strictly decreasing; one uses that and the given filtration of $C$ to define a filtration of $D \otimes C$, which is the only nontrivial part of checking that this is a twisted complex. Another useful operation on twisted complexes is to start with $a \in \hom_{\Tw(\scrA)}^0(C_0,C_1)$ which is closed (meaning $\mu^1_{\mathit{Tw}(\scrA)}(a) = 0$), and form its mapping cone $\mathit{Cone}(a) \in \Ob(\mathit{Tw}(\scrA))$. We will need a particular combination of these two processes later on. Namely, let $C_0$ and $C_1$ be twisted complexes, such that the complex $\hom_{\Tw(\scrA)}(C_0,C_1)$ is finite-dimensional. Then there is a canonical evaluation morphism $\mathit{ev}: \hom_{\Tw(\scrA)}(C_0,C_1) \otimes C_0 \rightarrow C_1$, which a closed morphism in $\Tw(\scrA)$. By twisting $C_1$ along $C_0$, we mean forming the twisted complex
\begin{equation} \label{eq:t-twisted}
T_{C_0}(C_1) = \mathit{Cone}(\mathit{ev}).
\end{equation}

Any $A_\infty$-category has an intrinsic notion of exact triangle \cite[Section 3]{seidel04}. As an example, take a twisted complex $C_1$ and a subcomplex $C_0 \subset C_1$, defined by taking subspaces of the vector spaces in \eqref{eq:formal-sum} in a way which is compatible with the differential. Correspondingly, one has a quotient twisted complex $C_2 = C_1/C_0$, and they form an exact triangle
\begin{equation} \label{eq:triangle}
\xymatrix{
C_0 \ar[rr] && C_1 \ar[d] \\
&& C_2 \ar[ull]^-{[1]}
}
\end{equation}
where the connecting homomorphism (marked by $[1]$ in our notation, because it has degree $1$) is canonical in $H^*(\Tw(\scrA))$. By applying this repeatedly, one can find a collection of exact triangles which decompose any twisted complex into objects of $\scrA$ (up to shifts). Conversely, one recovers $C_1$ as the mapping cone of the morphism $C_2 \rightarrow C_0[1]$, so any twisted complex can be built up from objects of $\scrA$ by (shifts and) forming repeated mapping cones.

We also want to consider the $A_\infty$-category $\Mod(\scrA)$ of (strictly unital, right) $A_\infty$-modules over $\scrA$. Such an $A_\infty$-module $\scrM$ assigns to each $X \in \Ob(\scrA)$ a graded vector space $\scrM(X)$, together with operations
\begin{equation}
\mu_{\scrM}^{d+1}: \scrM(X_d) \otimes \hom_{\scrA}(X_{d-1},X_d) \otimes \cdots \otimes \hom_{\scrA}(X_0,X_1) \longrightarrow \scrM(X_0)[1-d].
\end{equation}
In particular, $\scrM(X)$ is a chain complex with differential $\mu^1_{\scrM}$, up to a change of sign as in \eqref{eq:signed-differential}. $A_\infty$-modules admit analogues of the constructions discussed above (tensor product with a chain complex of vector spaces, which in this case can be arbitrary; and mapping cone). There is a canonical cohomologically full and faithful $A_\infty$-functor
\begin{equation} \label{eq:yoneda}
\scrA \longrightarrow \Mod(\scrA),
\end{equation}
the Yoneda embedding, which takes an object $Y$ to an $A_\infty$-module $\scrY$ with $\scrY(X) = \hom_{\scrA}(X,Y)$. One can extend \eqref{eq:yoneda} to a cohomologically full and faithful embedding of $\Tw(\scrA)$, either by directly generalizing the definition, or (equivalently) by using the composition
\begin{equation} \label{eq:tw-yoneda}
\Tw(\scrA) \xrightarrow{\text{Yoneda}} \Mod(\Tw(\scrA)) \xrightarrow{\text{restriction}} \Mod(\scrA).
\end{equation}

\subsection{Directedness}
A particularly well-behaved class of $A_\infty$-categories are the directed ones (these were first introduced by Kontsevich \cite{kontsevich98}, based on the theory of exceptional collections and its dg refinement \cite{bondal-kapranov91}). Let $\scrA$ be an $A_\infty$-category with finitely many ordered objects, $\Ob(\scrA) = \{X^1,\dots,X^m\}$. One says that $\scrA$ is directed if
\begin{equation} \label{eq:directed}
\hom_{\scrA}(X^i,X^j) = \begin{cases} \text{finite-dimensional} & i<j, \\
\bC \, e_{X^i} & i = j, \\ 0 & i>j. \end{cases}
\end{equation}
Let $\scrB$ be an $A_\infty$-category such that $\mathit{hom}_{\scrB}(X_0,X_1)$ is finite-dimensional for all $(X_0,X_1)$. Take an ordered collection $(X^1,\dots,X^m)$ of objects in it. One defines the associated directed $A_\infty$-subcategory $\scrA$ by taking the $X^i$ as objects, and essentially imposing \eqref{eq:directed} by brute force:
\begin{equation}
\hom_{\scrA}(X^i,X^j) = \begin{cases} \hom_{\scrB}(X^i,X^j) & i<j, \\ \bC \, e_{X^i} & i = j,
\\ 0 & i>j, \end{cases}
\end{equation}
with the $A_\infty$-structure inherited from $\scrB$ (and such that the $e_{X^i}$ are strict units). One can generalize this construction to the case where $\scrB$ is proper, but it won't be strictly unique anymore (and we don't need the generalization for our applications).

\begin{lemma} \label{th:proper-vs-strictly-proper}
Let $\scrA$ be a directed $A_\infty$-category, and $\scrM_0$ an $A_\infty$-module over $\scrA$ such that for each $k$, $H^*(\scrM_0(X^k))$ is finite-dimensional. Then there is a quasi-isomorphic $A_\infty$-module $\scrM_1$ such that $\scrM_1(X^k)$ is finite-dimensional for all $k$.
\end{lemma}

\begin{proof}
We will use the following basic fact about chain complexes (of vector spaces): if $D_0$ is a complex such that $H^*(D_0)$ is finite-dimensional, and $D \subset D_0$ an arbitrary finite-dimensional graded subspace, one can choose a finite-dimensional subcomplex $D_1$ which contains $D$, and such that the inclusion $D_1 \rightarrow D_0$ is a quasi-isomorphism.

This fact is relevant here for the following reason. Take $\scrM_0$ as in the statement. By descending induction on $k$, one can choose finite-dimensional subcomplexes $\scrM_1(X^k) \subset \scrM_0(X^k)$ such that the inclusion is a quasi-isomorphism, and such that $\scrM_1(X^k)$ contains the image of the (finitely many) operations
\begin{equation} \label{eq:module-maps}
\mu_{\scrM}^{d+1}: \scrM_1(X^{k_d}) \otimes \hom_{\scrA}(X^{k_{d-1}},X^{k_d}) \otimes \cdots \otimes \hom_{\scrA}(X^k,X^{k_1}) \longrightarrow \scrM_1(X^k)[1-d]
\end{equation}
with $d>0$ and $k < k_1 < \cdots < k_d$. By construction, the resulting $\scrM_1$ is an $A_\infty$-submodule quasi-isomorphic to $\scrM_0$.
\end{proof}

\begin{lemma} \label{th:module-vs-twisted}
Let $\scrA$ be a directed $A_\infty$-category, and $\scrM$ an $A_\infty$-module such that for each $k$, $\scrM(X^k)$ is finite-dimensional. Then there is a twisted complex whose image under \eqref{eq:tw-yoneda} is quasi-isomorphic to $\scrM$.
\end{lemma}

\begin{proof}
We will use the twist operation along $X^k$, from \cite[Section 5a]{seidel04}, which is an analogue of \eqref{eq:t-twisted} for $A_\infty$-modules. If $\scrM$ is an $A_\infty$-module, then its image under twisting, $\scrT = T_{X^k}(\scrM)$, is another $A_\infty$-module with
\begin{equation}
\begin{aligned}
& \scrT(X^j) = (\scrM(X^k) \otimes \hom_{\scrA}(X^j,X^k)) [1] \oplus \scrM(X^j),
\\
&
\label{eq:twist-mu1}
\mu^1_{\scrT}(m_1 \otimes a, m_0) = \big((-1)^{|a|-1} \mu^1_{\scrM}(m_1) \otimes a +
m_1 \otimes \mu^1_{\scrA}(a), \mu^1_{\scrM}(m_0) + \mu^2_{\scrM}(m_1,a)\big),
\end{aligned}
\end{equation}
and similar formulae for $\mu^{d+1}_{\scrT}$ \cite[Equation (5.1)]{seidel04}. By construction, this fits into an exact triangle \cite[Equation (5.3)]{seidel04}
\begin{equation} \label{eq:tw-exact-triangle}
\xymatrix{
\scrM(X^k) \otimes \scrX^k \ar[rr] && \scrM \ar[d] \\ 
&& \scrT = T_{X^k}(\scrM). \ar[ull]^-{[1]}
}
\end{equation}
The left corner of \eqref{eq:tw-exact-triangle} consists of the Yoneda module $\scrX^k$ associated to $X^k$, tensored with the finite-dimensional chain complex $\scrM(X^k)$. Equivalently, one can consider $\scrM(X^k) \otimes X^k$ itself as a twisted complex, and then $\scrM(X^k) \otimes \scrX^k$ is the image of that complex under \eqref{eq:tw-yoneda}.

By tracking the cohomologies of \eqref{eq:twist-mu1}, one sees that $T_{X^1} \cdots T_{X^m}(\scrM)$ is always acyclic, hence quasi-isomorphic to zero (compare \cite[Lemma 5.13]{seidel04}). On the other hand, one can generalize \eqref{eq:tw-exact-triangle} to a sequence of Dehn twists, and obtain an exact triangle 
\begin{equation} \label{eq:c-triangle}
\xymatrix{
\scrC \ar[rr] && \scrM \ar[d] \\
&& T_{X^1} \cdots T_{X^m}(\scrM).
\ar[ull]^-{[1]}
}
\end{equation}
Here,
\begin{equation} \label{eq:scrc}
\scrC(X^j) = \bigoplus\; \scrM(X^{j_d}) \otimes \hom_{\scrA}(X^{j_{d-1}},X^{j_d}) \otimes \cdots 
\otimes \hom_{\scrA}(X^{j},X^{j_1})[d-1], 
\end{equation}
where the (finite) sum is over all $d > 0$ and $j < j_1 < \cdots < j_d$. The differential is a form of the bar differential:
\begin{equation} \label{eq:c-mu1}
\begin{aligned}
\mu_{\scrC}^1(m \otimes a_d \otimes \cdots \otimes a_1) & =
\sum_i (-1)^{|a_1|+\cdots+|a_i|-i} \mu_{\scrM}^{d-i+1}(m,a_d,\dots,a_{i+1}) \otimes \cdots \otimes a_1 \\ & +
\sum_{i,j} (-1)^{|a_1|+|\cdots| + |a_i|-i} m \otimes \cdots \otimes \mu_{\scrA}^j(a_{i+j},\dots,a_{i+1}) \otimes \cdots \otimes a_1,
\end{aligned}
\end{equation}
and the higher $A_\infty$-module structure is defined in a way similar to the second term in \eqref{eq:c-mu1}. In fact, $\scrC$ is the Yoneda image of a twisted complex, of the form
\begin{equation} \label{eq:c}
C = \bigoplus\; \scrM(X^{j_d}) \otimes \hom_{\scrA}(X^{j_{d-1}},X^{j_d}) \otimes \cdots \otimes \hom_{\scrA}(X^{j_1},X^{j_2}) \otimes X^{j}[d-1].
\end{equation}
We omit the precise form of the differential on $C$, which can easily be inferred from \eqref{eq:c-mu1} (for an appearance of similar formulae elsewhere in the literature, see e.g.\ \cite[Equation (6.10)]{maydanskiy-seidel09}). Finally, since $T_{X^1}\cdots T_{X^m}(\scrM)$ is quasi-isomorphic to zero, the horizontal arrow in \eqref{eq:c-triangle} is necessarily a quasi-isomorphism.
\end{proof}

What is of interest to us is the combination of the previous two Lemmas, which shows that for a directed $A_\infty$-categories, $A_\infty$-modules (with a suitable properness condition) are essentially the same as twisted complexes:

\begin{proposition} \label{th:directed-module}
Let $\scrA$ be a directed $A_\infty$-category, and $\scrM$ an $A_\infty$-module over it such that $H^*(\scrM(X^k))$ is finite-dimensional for all $k$. Then there is a twisted complex whose image under the Yoneda embedding is quasi-isomorphic to $\scrM$.
\end{proposition}

\subsection{Hochschild (co)homology}
This final batch of background material will only have minor importance for us, and therefore, the exposition will be particularly terse. The Hochschild homology $\mathit{HH}_*(\scrA,\scrA)$ and Hochschild cohomology $\mathit{HH}^*(\scrA,\scrA)$ of an $A_\infty$-category are obtained from chain complexes of the form
\begin{align} \label{eq:hh}
& \mathit{CC}_*(\scrA,\scrA) = \bigoplus \hom_{\scrA}(X_{d-1},X_d) \otimes \hom_{\scrA}(X_1,X_2) \otimes \hom_{\scrA}(X_0,X_1) \otimes \hom_{\scrA}(X_d,X_0)[d],
\\ \label{eq:hh-2}
& \mathit{CC}^*(\scrA,\scrA) = \prod \mathit{Hom}\big(\hom_{\scrA}(X_{d-1},X_d) \otimes \cdots \otimes \hom_{\scrA}(X_0,X_1), \hom_{\scrA}(X_0,X_d)\big)[d],
\end{align}
where the direct sum and product are over all $d \geq 0$ and collections of objects $(X_0,\dots,X_d)$. We will not write down the differentials in full, but we want to note the first couple of equations for the components $(\phi^0, \phi^1,\dots)$ of a degree $1$ Hochschild cocycle $\phi \in \mathit{CC}^1(\scrA,\scrA)$:
\begin{equation}
\begin{aligned}
& \mu^1_{\scrA}(\phi^0_X) = 0, \\
& \mu^2_{\scrA}(\phi^0_{X_1},a) + \mu^2_{\scrA}(a,\phi^0_{X_0}) + \mu^1_{\scrA}(\phi^1_{X_0,X_1}(a)) - \phi^1_{X_0,X_1}(\mu^1_{\scrA}(a)) = 0, \\
& \mu^2_{\scrA}(\phi^1_{X_1,X_2}(a_2),a_1) + \mu^2_{\scrA}(a_2,\phi^1_{X_0,X_1}(a_1)) - \phi^1_{X_0,X_2}(\mu^2_{\scrA}(a_2,a_1)) \\ & \qquad \qquad \qquad \qquad \qquad \qquad \qquad \qquad \quad + \text{terms involving $\mu^1_{\scrA}$} = 0,
\\ & \dots
\end{aligned}
\end{equation}
In particular, $\phi^0$ gives rise to a cohomology class 
\begin{equation}
[\phi^0_X] \in H^1(\hom_{\scrA}(X,X))
\end{equation}
for any $X$. Next, given two objects such that $\phi^0_{X_0} = 0$ and $\phi^0_{X_1} = 0$, we get an endomorphism of $H^*(\hom_{\scrA}(X_0,X_1))$, namely
\begin{equation}
[a] \longmapsto [\phi^1_{X_0,X_1}(a)],
\end{equation}
which is a derivation with respect to the product induced by \eqref{eq:signed-multiplication}.

Hochschild cohomology has the structure of a Gerstenhaber algebra, which consists of a graded commutative product on $\mathit{HH}^*(\scrA,\scrA)$, and a graded Lie bracket on $\mathit{HH}^*(\scrA,\scrA)[1]$. Hochschild homology is correspondingly both a module (over Hochschild cohomology as an algebra) and a Lie module (over shifted Hochschild cohomology as a Lie algebra). For instance, suppose that we have a cochain $\phi \in \mathit{CC}^1(\scrA,\scrA)$, whose only nonzero component is the linear part $\phi^1$. Then, its Lie action is simply the endomorphism
\begin{equation} \label{eq:lie-action}
\begin{aligned}
& 
\mathit{CC}_*(\scrA,\scrA) \longrightarrow \mathit{CC}_*(\scrA,\scrA), \\
&
a_d \otimes \cdots \otimes a_0 \longmapsto \sum_i a_d \otimes \cdots \otimes \phi^1(a_i) \otimes \cdots \otimes a_0.
\end{aligned}
\end{equation}
Additionally, Hochschild homology carries the Connes operator, an endomorphism of degree $-1$ (note that in our convention, the grading of $\mathit{HH}_*(\scrA,\scrA)$ is cohomological, in spite of the notation).

Suppose that $\scrA$ is proper. Then, one defines a weak Calabi-Yau structure on $\scrA$ of dimension $n \in \bZ$ to be a quasi-isomorphism of $A_\infty$-bimodules
\begin{equation} \label{eq:weak-cy}
\scrA \stackrel{\iso}{\longrightarrow} \scrA^\vee[-n]
\end{equation}
(compare \cite{tradler01}, which introduces a slightly more restrictive version). On the cohomology level, the weak Calabi-Yau structure induces nondegenerate pairings, which we denote by
\begin{equation} \label{eq:cy-pairings}
\begin{aligned}
& H^*(\hom_{\scrA}(X_1,X_0)) \otimes H^*(\hom_{\scrA}(X_0,X_1)) \longrightarrow \bC[-n], \\
& [a_1] \otimes [a_0] \longmapsto \langle [a_1], [a_0] \rangle_{\mathit{CY}}.
\end{aligned}
\end{equation}
These satisfy
\begin{equation}
\langle [a_2] \cdot [a_1], [a_0] \rangle_{\mathit{CY}} = 
\langle [a_2], [a_1] \cdot [a_0] \rangle_{\mathit{CY}} =
(-1)^{|a_2|(|a_1|+|a_0|)} \langle [a_1], [a_0] \cdot [a_2] \rangle_{\mathit{CY}}.
\end{equation}

A weak Calabi-Yau structure gives rise to an isomorphism 
\begin{equation}
\mathit{HH}^{*+n}(\scrA,\scrA) \iso \mathit{HH}_*(\scrA,\scrA)^\vee.
\end{equation}
The dual of the BV operator then carries over to an endomorphism of $\mathit{HH}^*(\scrA,\scrA)$ of degree $-1$, which we denote by $\Delta_{CY}$. Again, we omit the full formulae, and concentrate on the simplest piece. Suppose that we have cocycles $\phi \in \mathit{CC}^1(\scrA,\scrA)$ and $\psi \in \mathit{CC}^0(\scrA,\scrA)$, such that $[\psi] = \Delta_{CY} [\phi] \in \mathit{HH}^0(\scrA,\scrA)$. Then
\begin{equation} \label{eq:contraction}
\langle [e_X], [\phi^1_{X,X}(a)] \rangle_{\mathit{CY}} = \langle [\psi^0_X], [a] \rangle_{\mathit{CY}}
\end{equation}
for any object $X$ such that $\phi^0_X = 0$, and any $[a] \in H^n(\hom_{\scrA}(X,X))$.

\section{Algebraic notions\label{sec:tools}}

\subsection{Naive $\bC^*$-actions\label{subsec:naive-actions}}
Recall that a rational representation of $\bC^*$ is a representation which is a direct sum of finite-dimensional ones. Equivalently, these are the representations of the form 
\begin{equation} \label{eq:weight-sum}
W = \bigoplus_{i \in \bZ} W^i, 
\end{equation}
where $\bC^*$ acts on $W^i$ with weight $i$. The decomposition \eqref{eq:weight-sum} is canonical, so the structure of a rational $\bC^*$-representation is essentially the same as a grading on $W$, but the representation-theoretic language is more natural in our context. We denote by 
\begin{equation} \label{eq:langle}
W \longmapsto W\langle k \rangle
\end{equation}
the operation of tensoring a given representation with the one-dimensional weight $k$ representation. Equivalently, this consists of shifting the indexing of \eqref{eq:weight-sum} up by $k$.

\begin{definition} \label{th:naive-action}
An $A_\infty$-category with a naive $\bC^*$-action is an $A_\infty$-category $\scrA$ where each vector space $\hom_{\scrA}^k(X_0,X_1)$ is a rational representation of $\bC^*$, in such a way that the $\mu_{\scrA}^d$ are equivariant, and the strict identities are $\bC^*$-invariant. We denote by $\scrA^{\bC^*}$ the category with the same objects, but where one retains only the invariant part of the morphism spaces.
\end{definition}

A naive $\bC^*$-action gives rise to a cocycle $\mathit{def} \in \mathit{CC}^1(\scrA,\scrA)$, whose only nonzero term $\mathit{def}^1$ is the infinitesimal generator of the action (the endomorphism that multiplies the weight $i$ part of $\hom_{\scrA}(X_0,X_1)$ by $i$). Moreover, $\mathit{CC}_*(\scrA,\scrA)$ is a rational representation of $\bC^*$, and the infinitesimal generator of that is the Lie action of $\mathit{def}$, in the sense of \eqref{eq:lie-action}. On the level of cohomology, we therefore get a class $\mathit{Def} = [\mathit{def}] \in \mathit{HH}^1(\scrA,\scrA)$, whose Lie action on $\mathit{HH}_*(\scrA,\scrA)$ is the infinitesimal generator of a rational $\bC^*$-representation.

The previously mentioned Hochschild cocycle in $\scrA$ gives rise to one in $\Tw(\scrA)$, for which we use the same notation $\mathit{def}$. It has two nonzero components
\begin{equation}
\begin{aligned}
& \mathit{def}^0_C \in \hom^1_{\Tw(\scrA)}(C,C), \\
& \mathit{def}^1_{C_0,C_1}: \hom_{\Tw(\scrA)}(C_0,C_1) \longrightarrow \hom_{\Tw(\scrA)}(C_0,C_1).
\end{aligned}
\end{equation}
To describe these, start by introducing a $\bC^*$-action on $\hom_{\Tw(\scrA)}(C_0,C_1)$, by using the given action on the morphisms in $\scrA$ together with the trivial action on the $\mathit{Hom}$ spaces in \eqref{eq:formal-sum-morphism}. The associated infinitesimal action is $\mathit{def}^1_{C_0,C_1}$, and the remaining component $\mathit{def}^0_C$ is obtained by applying the infinitesimal action to $\delta_C$. We will be particularly interested in the constant components
\begin{equation} \label{eq:alg-def-class}
\mathit{Def}^0_C = [\mathit{def}^0_C] \in H^1(\hom_{\Tw(\scrA)}(C,C)).
\end{equation}

\subsection{Equivariant twisted complexes}
There is an equivariant version of $\Tw(\scrA)$ for an $A_\infty$-category $\scrA$ with a naive $\bC^*$-action, which is constructed as follows. When forming the equivariant analogue of the additive enlargement, one asks that the $W_f$ in \eqref{eq:formal-sum} should be finite-dimensional graded representations of $\bC^*$. The resulting $A_\infty$-category then inherits a naive $\bC^*$-action, given by taking the tensor product representation in \eqref{eq:formal-sum-morphism}. When introducing equivariant twisted complexes, one asks that the differential $\delta_C$ should lie in the $\bC^*$-invariant part of the endomorphism space. The outcome is another $A_\infty$-category with a naive $\bC^*$-action, denoted by $\EqTw(\scrA)$. It admits an operation of tensoring an object with a finite-dimensional chain complex of $\bC^*$-representations. In particular, one can tensor with one-dimensional representations as in \eqref{eq:langle}, and we use the same notation for it,
\begin{equation} \label{eq:l-twist}
C \longmapsto C\langle k \rangle.
\end{equation}
Equivariant twisted complexes also admit mapping cones with respect to $\bC^*$-invariant closed morphisms. The forgetful functor $\EqTw(\scrA) \rightarrow \Tw(\scrA)$ is full and faithful by definition, but by no means a quasi-equivalence, as the following observation shows:
%
%

\begin{lemma} \label{th:obstruction}
If a twisted complex $C$ is quasi-isomorphic to an equivariant twisted complex, the class \eqref{eq:alg-def-class} vanishes.
\end{lemma}

\begin{proof}
Because it comes from a Hochschild cohomology class, \eqref{eq:alg-def-class} is central, in the sense that
\begin{equation}
\mathit{Def}^0_{C_1} \cdot [a] = (-1)^{|a|} [a] \cdot \mathit{Def}^0_{C_0}
\quad \text{for } [a] \in H^*(\hom_{\Tw(\scrA)}(C_0,C_1)).
\end{equation}
In particular, it is a quasi-isomorphism invariant. 

Suppose that we have an equivariant twisted complex $C$, written as in \eqref{eq:formal-sum}. Take the infinitesimal $\bC^*$-action on each space $W_f$, denoted by $\xi_f$, and form
\begin{equation}
\xi_C = \bigoplus_f \xi_f \otimes e_{X_f} \in \hom_{\Tw(\scrA)}^0(C,C).
\end{equation}
Because $\delta_C$ is $\bC^*$-invariant, we have
\begin{equation}
\mathit{def}^1_{C,C}(\delta_C) - \mu^2_{\scrA^\oplus}(\xi_C,\delta_C) - \mu^2_{\scrA^\oplus}(\delta_C,\xi_C) = 0,
\end{equation}
which one can rewrite as $\mathit{def}^0_C = \mu^1_{\Tw(\scrA)}(\xi_C)$. Hence \eqref{eq:alg-def-class} vanishes in this case.
\end{proof}

Our principal interest is in the converse direction. The argument for that will use $A_\infty$-modules as a stepping-stone. Let $\scrA$ be an $A_\infty$-category with a naive $\bC^*$-action. An equivariant $A_\infty$-module over $\scrA$ (in the naive sense) assigns to $X \in \Ob(\scrA)$ a space $\scrM(X)$, each graded piece of which is a rational representation of $\bC^*$, together with structure maps as in \eqref{eq:module-maps}, which are $\bC^*$-equivariant. Given an equivariant twisted complex, its image under \eqref{eq:tw-yoneda} is naturally an equivariant module. Equivariant modules, together with the same morphisms as in the non-equivariant case, form an $A_\infty$-category $\EqMod(\scrA)$. It is important to note that the spaces
\begin{multline} \label{eq:module-hom}
\hom_{\EqMod(\scrA)}(\scrM_0,\scrM_1) = \prod \mathit{Hom}(\scrM_0(X_d) \otimes \hom_{\scrA}(X_{d-1},X_d) \otimes \cdots \\ \cdots \otimes \hom_{\scrA}(X_0,X_1), \scrM_1(X_0))[d]
\end{multline}
(where the product is over all $d \geq 0$ and objects $X_0,\dots,X_d$) carry induced $\bC^*$-action, which however are not necessarily rational representations. Hence, $\EqMod(\scrA)$ does not satisfy the conditions of Definition \ref{th:naive-action}. 
%

\begin{proposition} \label{th:equivariant-module}
Let $\scrA$ be a directed $A_\infty$-category with a naive $\bC^*$-action. Let $\scrM$ be an equivariant $A_\infty$-module such that $H^*(\scrM(X^k))$ is finite-dimensional for each $k$. Then there is an equivariant twisted complex whose image under \eqref{eq:tw-yoneda} is quasi-isomorphic to $\scrM$.
\end{proposition}

This is the analogue of Proposition \ref{th:directed-module}, and the proof remains the same. There are only two noteworthy technical points: first, the property of chain complexes used in Lemma \ref{th:proper-vs-strictly-proper} also holds in the presence of a rational $\bC^*$-action (simply by treating the summands \eqref{eq:weight-sum} one at a time). Secondly, if $\scrM$ is an equivariant $A_\infty$-module, then so is $T_{X^k}\scrM$ for any $k$. Hence, if one starts with an equivariant $\scrM$, all the modules in \eqref{eq:c-triangle} will be equivariant, and so is the twisted complex \eqref{eq:c}. The desired converse to Lemma \ref{th:obstruction} is:

\begin{theorem} \label{th:rigid-lifting}
Let $\scrA$ be a directed $A_\infty$-category with a naive $\bC^*$-action. Let $C_1$ be a twisted complex such that \eqref{eq:alg-def-class} vanishes, and with $H^0(\hom_{\Tw(\scrA)}(C_1,C_1)) \iso \bC$. Then there is an equivariant twisted complex $C_0$ which is quasi-isomorphic to $C_1$.
\end{theorem}

\begin{proof}
Consider the Yoneda image $\scrC_1$ of $C_1$. One can apply \cite[Lemma 7.12]{seidel12} to equip this with a weak $\bC^*$-action, in the sense of \cite[Definition 7.7]{seidel12}. Note that while \cite[Lemma 7.12]{seidel12} makes the assumption that
\begin{equation} \label{eq:h1}
H^1(\hom_{\Mod(\scrA)}(\scrC_1,\scrC_1)) = H^1(\hom_{\Tw(\scrA)}(C_1,C_1)) = 0, 
\end{equation}
the only thing needed in the proof is the vanishing of a certain cohomology class \cite[Equation (7-3)]{seidel12}, which is the image of \eqref{eq:alg-def-class} under the Yoneda embedding. Applying \cite[Lemma 8.3]{seidel12} upgrades the weak action to a homotopy action \cite[Definition 8.1]{seidel12}, and by \cite[Lemma 8.2]{seidel12} this implies the existence of a quasi-isomorphic module $\scrC_0$ which is equivariant in the sense considered in this paper. (The whole process is combined into one step in \cite[Corollary 8.4]{seidel12}, where the assumption \eqref{eq:h1} should be weakened as before.)

By construction, $H^*(\scrC_0(X^k)) \iso H^*(\scrC_1(X^k)) \iso H^*(\hom_{\Tw(\scrA)}(X^k,C_1))$ is finite-dimensional for each $k$. Hence, $\scrC_0$ is quasi-isomorphic to the Yoneda image of some equivariant twisted complex $C_0$, by Proposition \ref{th:equivariant-module}.
\end{proof}

\begin{remark}
The point of the detour via $A_\infty$-modules is that the operations \eqref{eq:module-maps} can be constructed order by order using a suitable obstruction theory, which is indeed the strategy used in \cite[Lemma 8.3]{seidel12}. The disadvantage of this is that equivariant $A_\infty$-modules are only well-behaved in certain cases; happily, directed $A_\infty$-categories are one of those cases.
\end{remark}

\subsection{Grothendieck groups}
This final piece of algebraic machinery serves as a point of transition to our applications. For an $A_\infty$-category $\scrA$, one defines $K_0(\scrA)$ to be the group generated by the quasi-isomorphism classes of twisted complexes, with the relation that $[C_0]-[C_1]+[C_2] = 0$ if there is exact triangle of the form \eqref{eq:triangle}. In fact, $K_0(\scrA)$ is already generated by the classes of the objects of $\scrA$ itself. Suppose now that $\scrA$ is proper; then, so is $\Tw(\scrA)$. Define the Mukai pairing between twisted complexes to be the Euler characteristic
\begin{equation}
C_0 \cdot C_1 = \chi(H^*(\hom_{\Tw(\scrA)}(C_0,C_1))) \in \bZ.
\end{equation}
This descends to a bilinear pairing on $K_0(\scrA)$, because an exact triangle induces long exact sequences of (cohomology level) morphism spaces. For instance, in the case where $\scrA$ is directed and has $m$ objects, the Mukai pairing between these objects is nondegenerate, which implies that $K_0(\scrA) \iso \bZ^m$.

We now introduce the equivariant analogues of these notions. Take an $A_\infty$-category $\scrA$ with a naive $\bC^*$-action. One defines the equivariant Grothendieck group $K_0^{\bC^*}(\scrA)$ to be the ordinary Grothendieck group of $\EqTw(\scrA)^{\bC^*}$ (recall that the superscript means that only invariant morphisms are allowed). $K_0^{\bC^*}(\scrA)$ is naturally a module over $\bZ[q,q^{-1}]$, with multiplication by $q^k$ corresponding to \eqref{eq:l-twist}. Moreover, over $\bZ[q,q^{-1}]$ it is generated by the classes of objects of $\scrA$ itself. If we suppose that $\scrA$ is proper, we can define the equivariant Mukai pairing to be 
\begin{equation} \label{eq:equivariant-mukai}
C_0 \cdot_q C_1 = \sum_k q^k \chi(H^*(\hom_{\EqTw(\scrA)^{\bC^*}}(C_0\langle k \rangle,C_1))) \in \bZ[q,q^{-1}].
\end{equation}
Equivalently, one can consider the space $H^*(\hom_{\EqTw(\scrA)}(C_0,C_1))$ with its induced $\bC^*$-action, and define $C_0 \cdot_q C_1$ as the Lefschetz trace of that action:
\begin{equation}
C_0 \cdot_q C_1 = \mathrm{Str}\big(q: H^*(\hom_{\EqTw(\scrA)}(C_0,C_1)) \longrightarrow
H^*(\hom_{\EqTw(\scrA)}(C_0,C_1)) \big).
\end{equation}
(In the application to symplectic topology, this will give rise to the two equivalent definitions, \eqref{eq:q-intersection-2} and \eqref{eq:q-intersection}, of $q$-intersection number). As before, \eqref{eq:equivariant-mukai} descends to a pairing on $K_0^{\bC^*}(\scrA)$, and this satisfies
\begin{equation}
(q^{k_0}c_0) \cdot_q (q^{k_1}c_1) = q^{k_1-k_0} (c_0 \cdot_q c_1).
\end{equation}
Here is a sample application: suppose that $C_0,C_1$ are equivariant twisted complexes, such that $\hom_{\Tw(\scrA)}(C_0,C_1)$ is finite-dimensional. Then the outcome of twisting \eqref{eq:t-twisted} is again an equivariant twisted complex, and the exact triangle it sits in belongs to $\EqTw(\scrA)^{\bC^*}$; hence,
\begin{equation} \label{eq:twisted-k-theory}
[T_{C_0}(C_1)] = [C_1] - ([C_0] \cdot_q [C_1])\, [C_0] \in K_0^{\bC^*}(\scrA).
\end{equation}
(One can generalize both the twisting process and \eqref{eq:twisted-k-theory} so that no assumption besides the properness of $\scrA$ is needed, but we will not use that.) As a consequence, we have for any $C_2$
\begin{equation} \label{eq:alg-q-picard-lefschetz}
C_2 \cdot_q T_{C_0}(C_1) = C_2 \cdot_q C_1 - (C_2 \cdot_q C_0)(C_0 \cdot_q C_1).
\end{equation}
(Later on, this will yield \eqref{eq:q-picard-lefschetz}, the Picard-Lefschetz formula for $q$-intersection numbers.) Finally, we would like to return to the case of a directed $\scrA$, now assuming that it carries a naive $\bC^*$-action. Define $A_q \in \mathit{Mat}(m \times m,\bZ[q,q^{-1}])$ by 
\begin{equation}
A_{q,ij} = X^i \cdot_q X^j.
\end{equation}
As in the non-equivariant case, this is invertible, which implies that
\begin{equation}
K_0^{\bC^*}(\scrA) \iso \bZ[q,q^{-1}]^m.
\end{equation}
If $\scrA$ is obtained as a directed subcategory of an $A_\infty$-category $\scrB$ with a naive $\bC^*$-action, one can introduce a corresponding matrix $B_q$ for the equivariant Mukai pairings in $\scrB$, and the two are then related by \eqref{eq:q-a-matrix}.

\section{Fukaya categories\label{sec:fukaya}}

\subsection{$\bC^*$-actions and $q$-intersection numbers\label{subsec:g-intersection-numbers}}
We now apply the general theory to Fukaya categories, which leads to the formalism described previously in Section \ref{subsec:dilating}. 

\begin{definition} \label{th:fuk-has-an-action}
Let $M^{2n}$ be a symplectic manifold as in Setup \ref{th:setup-symplectic}. A $\bC^*$-action on $\scrF(M)$ is given by the following data: a directed $A_\infty$-category $\scrC$ which has a naive $\bC^*$-action (Definition \ref{th:naive-action}), and a cohomologically full and faithful $A_\infty$-functor $\scrI: \scrF(M) \longrightarrow \Tw(\scrC)$.
\end{definition}

\begin{definition} \label{th:equivariant-l}
Suppose that $\scrF(M)$ has a $\bC^*$-action. By an equivariant structure on an exact Lagrangian brane $L$, we mean the choice of an isomorphism $C \rightarrow \scrI(L)$ in $H^0(\Tw(\scrC))$, where $C$ is an equivariant twisted complex. Two such structures are considered to be equivalent if there is a commutative diagram of isomorphisms
\begin{equation}
\xymatrix{
C_0 \ar[d] \ar[dr] \\ C_1 \ar[r] & \scrI(L) 
}
\end{equation}
where the $\downarrow$ lies in the $\bC^*$-invariant part of $H^0(\hom_{\EqTw(\scrC)}(C_0,C_1))$.
\end{definition}

The deformation class \eqref{eq:l-def-class} of $L$ is the element corresponding to \eqref{eq:alg-def-class} under the isomorphism 
\begin{equation}
H^1(L;\bC) \iso \mathit{HF}^1(L,L) = H^1(\hom_{\scrF(M)}(L,L)) \iso H^1(\hom_{\Tw(\scrC)}(\scrI(L),\scrI(L))).
\end{equation}
By Theorem \ref{th:rigid-lifting}, any $L$ for which $\mathit{Def}^0_L = 0$ can be made equivariant. Clearly, if $L_0,L_1$ have been made equivariant, one gets a $\bC^*$-action on 
\begin{equation} \label{eq:l-c}
\mathit{HF}^*(L_0,L_1) \iso H^*(\hom_{\EqTw(\scrC)}(C_0,C_1)).
\end{equation}
Consider the case where $L_0 = L_1 = L$, made equivariant in two ways. The degree zero part of \eqref{eq:l-c} is one-dimensional, and if the action on it has weight $k$, then $C_0 \langle k \rangle $ is equivalent to $C_1$. Hence, equivariant structures up to equivalence form an affine space over $\bZ$. In a slight abuse of notation, we write \eqref{eq:equi-shift} for the change of equivariant structure. Property \eqref{eq:shift-equivariance} is obvious from the definitions. In particular, given any $L$ with $\mathit{Def}^0_L = 0$, one has a unique $\bC^*$-action on $H^*(L;\bC)$. This justifies the following:

\begin{definition} \label{th:dilating-action}
We say that the $\bC^*$-action on $\scrF(M)$ is dilating if, for any exact Lagrangian brane such that $\mathit{Def}_L^0 = 0$, the $\bC^*$-action on $\mathit{HF}^n(L,L) \iso H^n(L;\bC) \iso \bC$ has weight $1$.
\end{definition}

%

\begin{example} \label{th:cylinder-2}
We spell out the details of Example \ref{th:cylinder}. Any exact Lagrangian submanifold $L \subset M = [-1,1] \times S^1$ is Hamiltonian isotopic to $\{0\} \times S^1$. There are two choices of {\em Spin} structures, giving rise to two objects $L_0$, $L_1$ which are orthogonal: $\mathit{HF}^*(L_0,L_1) = 0$. The endomorphism ring of each object is an exterior algebra in one variable, which is intrinsically formal (it does not carry any nontrivial $A_\infty$-structures). 

Consider the Kronecker quiver $\scrC$, which is the unique directed $A_\infty$-category with two objects $(X^1,X^2)$, and $\mathit{hom}_{\scrC}(X^1,X^2) = W$ a two-dimensional vector space concentrated in degree $0$ (this completely determines the $A_\infty$-structure). Take the twisted complex $C = \mathit{Cone}(w: X^1 \rightarrow X^2)$, for some nonzero $w \in W$. Then
\begin{equation} \label{eq:endo-compute}
H^*(\hom_{\Tw(\scrC)}(C,C)) \iso \begin{cases} \bC \, [e_C] & \ast = 0, \\
W/(\bC \, w) & \ast = 1, \\
0 & \text{in all other degrees}
\end{cases} 
\end{equation}
is an exterior algebra in one variable. If one takes two disjoint copies of the Kronecker quiver (which can be considered as a directed $A_\infty$-category with four objects), its category of twisted complexes will therefore contain a full subcategory quasi-isomorphic to $\scrF(M)$.

Now let's turn $W$ into a representation of $\bC^*$, by letting it act trivially on our chosen $w$, and with weight one on a complementary one-dimensional subspace. $C$ is an equivariant twisted complex, and one sees from \eqref{eq:endo-compute} that the induced action on $H^1(\hom_{\Tw(\scrC)}(C,C))$ has weight $1$. Hence, one can use the previously mentioned embedding to equip $\scrF(M)$ with a dilating $\bC^*$-action.
\end{example}

Definition \ref{th:dilating-action} is formulated as a separate condition for each brane, which might seem to make it hard to check (except in very simple cases, such as Example \ref{th:cylinder-2}, where the branes can be classified explicitly). However, that impression is misleading, as the following shows:

\begin{lemma} \label{th:dilating-generators}
Suppose that $\scrF(M)$ has a $\bC^*$-action. Suppose also that there is a set $\{L^k\}$ of exact Lagrangian branes with the following properties:
\begin{itemize} \itemsep0.5em
\item Each $L^k$ can be made equivariant, and the $\bC^*$-action on $\mathit{HF}^n(L^k,L^k) \iso H^n(L^k;\bC) \iso \bC$ has weight $1$.
\item For any exact Lagrangian brane $L \subset M$, there is a $k$ such that $\mathit{HF}^*(L,L^k) \neq 0$.
\end{itemize}
Then the $\bC^*$-action is dilating.
\end{lemma}

\begin{proof}
From the definition, it is clear that the $\bC^*$-action on \eqref{eq:l-c} is compatible with composition in the Fukaya category. Hence, using the pairings \eqref{eq:two-sided-product} it follows that we have isomorphisms of $\bC^*$-representations
\begin{equation}
\mathit{HF}^{n-*}(L^k,L)^\vee \otimes \mathit{HF}^n(L^k,L^k) \iso \mathit{HF}^*(L,L^k) \iso
\mathit{HF}^{n-*}(L^k,L)^\vee \otimes \mathit{HF}^n(L,L).
\end{equation}
After choosing $k$ so that $\mathit{HF}^*(L,L^k) \neq 0$, it follows that the $\bC^*$-actions on $\mathit{HF}^n(L^k,L^k)$ and $\mathit{HF}^n(L,L)$ must have the same weight.
\end{proof}

The only property stated in Section \ref{subsec:dilating} which remains to be explained is \eqref{eq:q-picard-lefschetz}. Consider the setup of that equation, and where the $\scrI(L_k)$ are quasi-isomorphic to equivariant twisted complexes $C_k$. Recall \cite[Corollary 17.17]{seidel04} that there is a quasi-isomorphism 
\begin{equation} \label{eq:twist-twist}
T_{C_0}(C_1) \htp \scrI(\tau_{L_0}(L_1)).
\end{equation}
Since $T_{C_0}(C_1)$ is an equivariant twisted complex, one can use \eqref{eq:twist-twist} as an equivariant structure for $\tau_{L_0}(L_1)$. Then, \eqref{eq:alg-q-picard-lefschetz} implies \eqref{eq:q-picard-lefschetz}. As for \eqref{eq:q-picard-lefschetz-inverse}, one can either derive it from \eqref{eq:q-picard-lefschetz}, or give an independent but parallel argument for it.

\subsection{Discussion}
We add some considerations which, while less important for immediate applications, will help to clarify the meaning of Definitions \ref{th:fuk-has-an-action} and \ref{th:dilating-action}. Recall that the naive $\bC^*$-action on $\scrC$ gives rise to a Hochschild cohomology class for $\Tw(\scrC)$. Restrict this to the image of the embedding $\scrI$, and pull it back to the Fukaya category. We still denote the outcome by 
\begin{equation} \label{eq:def-fukaya}
\mathit{Def} \in \mathit{HH}^1(\scrF(M),\scrF(M)).
\end{equation} 

\begin{lemma} \label{th:exponentiating}
Suppose that $\scrF(M)$ admits a $\bC^*$-action, and is split-generated by Lagrangian branes which can be made equivariant. Then, the Lie action of \eqref{eq:def-fukaya} on $\mathit{HH}_*(\scrF(M),\scrF(M))$ is the infinitesimal generator of a rational representation of $\bC^*$ (which means, Hochschild homology decomposes into eigenspaces with integer eigenvalues).
\end{lemma}

\begin{proof}
Without affecting Hochschild homology or cohomology, we can replace $\scrF(M)$ by the full subcategory of Lagrangian branes that can be made equivariant. That subcategory is quasi-equivalent to a full subcategory of $\EqTw(\scrC)$. But for that subcategory, the corresponding result is obvious, as noted previously.
\end{proof}

The Fukaya category has a preferred weak Calabi-Yau structure, whose associated pairings \eqref{eq:cy-pairings} are those from \eqref{eq:two-sided-product}. Consider the resulting operator $\Delta_{\mathit{CY}}$ on $\mathit{HH}^*(\scrF(M),\scrF(M))$.

\begin{lemma} \label{th:dilations}
Suppose that $\scrF(M)$ admits a $\bC^*$-action, such that the class \eqref{eq:def-fukaya} satisfies
\begin{equation} \label{eq:quasi}
\Delta_{\mathit{CY}}(\mathit{Def} \cdot U) = U,
\end{equation}
for some $U \in \mathit{HH}^0(\scrF(M),\scrF(M))$ which is invertible (with respect to the commutative ring structure). Then that action is dilating.
\end{lemma}

\begin{proof}
Let's restrict the weak Calabi-Yau structure to the full subcategory of Lagrangian branes that can be made equivariant. This induces a corresponding weak Calabi-Yau structure on a full subcategory $\scrT \subset \EqTw(\scrC)$. We correspondingly restrict $U$, and choose a cocycle $u \in \mathit{CC}^0(\scrT,\scrT)$ representing it. We can apply \eqref{eq:contraction} to the product of $\mathit{def}$ and $u$. The outcome is that for any object $C$ and any endomorphism $[a]$ of degree $n$,
\begin{equation} \label{eq:contraction-2}
\langle [u_C^0], [\mathit{def}^1_{C,C}(a)] \rangle_{\mathit{CY}} = \langle [e_C], [\mathit{def}^1_{C,C}(a)] \cdot [u_C^0] \rangle_{\mathit{CY}} = \langle [u_C^0], [a] \rangle_{\mathit{CY}}.
\end{equation}
By assumption, $C$ corresponds to a Lagrangian brane, hence its endomorphism space is one-dimensional in degree $n$. Hence, \eqref{eq:contraction-2} shows that the $\bC^*$-action has weight $1$ in that degree. By carrying this over to the Fukaya category, one obtains the desired statement.
\end{proof}

Lemma \ref{th:dilations} allows us to compare Definition \ref{th:dilating-action} with the more geometric notion of ``dilation'' from \cite{seidel-solomon10}. Recall that there is an open-closed string map from the symplectic cohomology of $M$ to the Hochschild cohomology of its Fukaya category,
\begin{equation} \label{eq:open-closed}
\mathit{SH}^*(M) \longrightarrow \mathit{HH}^*(\scrF(M),\scrF(M)).
\end{equation}
Moreover, this relates the BV operator $\Delta$ on symplectic cohomology with the operator $\Delta_{\mathit{CY}}$ associated to the preferred weak Calabi-Yau structure of $\scrF(M)$. Hence, if we have a class in $\mathit{SH}^1(M)$ which is a dilation in the sense of \cite[Definition 4.1]{seidel-solomon10}, its image in Hochschild cohomology satisfies \eqref{eq:quasi}, with $U$ being the identity. In fact, a better geometric counterpart of \eqref{eq:quasi} is the notion of ``quasi-dilation'' from \cite[Part 4]{seidel13b}, which generalizes that of dilation.

\begin{remark}
Another way of explaining the appearance of $U$ in \eqref{eq:quasi} is as follows. The conclusion of Lemma \ref{th:dilations} would hold even if one replaced the geometrically given weak Calabi-Yau structure of $\scrF(M)$ by another one. Clearly, two such structures \eqref{eq:weak-cy} differ by an automorphism of the diagonal bimodule, which is the same as an invertible element of $\mathit{HH}^0(\scrF(M),\scrF(M))$. If we change \eqref{eq:weak-cy} by such an element $U$, the effect is to replace $\Delta_{\mathit{CY}}$ by $U^{-1} \Delta_{\mathit{CY}} U$.

In terms of mirror symmetry (as discussed in \cite[Section 1]{seidel-solomon10}), the choice of weak Calabi-Yau structure corresponds to that of a complex volume form on the mirror. Just as the Fukaya category comes with a preferred such structure, the mirror comes with a preferred complex volume form. A dilation corresponds roughly to a holomorphic vector field which is expanding for that volume form; whereas a quasi-dilation would be a vector field which is expanding for {\em some} volume form.
\end{remark}

There are further differences between the approach taken here and that in \cite{seidel-solomon10}. For one thing, there are $\bC^*$-actions which ``do not come from geometry'' on an infinitesimal level, in the sense that the element $\mathit{Def} \in \mathit{HH}^1(\scrF(M),\scrF(M))$ does not lie in the image of \eqref{eq:open-closed} (the $\bC^*$-action from Example \ref{th:cylinder-2} is one of them). A more substantial difference is that \cite{seidel-solomon10} allowed arbitrary infinitesimal symmetries; clearly, not all of them integrate to $\bC^*$-actions (Lemma \ref{th:exponentiating} gives a necessary, but not sufficient, criterion).

\subsection{Fukaya categories of Lefschetz fibrations}
We work in Setup \ref{th:symplectic-lefschetz}. Choose a basis of vanishing paths $(\gamma^k)$, with its associated Lefschetz thimbles $(\Delta^k)$ and vanishing cycles $(V^k)$. Take the full $A_\infty$-subcategory of $\scrF(F)$ with objects $V^k$. By a standard algebraic process \cite[Lemma 2.1]{seidel04}, one can find a quasi-isomorphic $A_\infty$-category, which is strictly unital (and retains the property that the $\hom$ spaces are finite-dimensional). Let $\scrA$ be the resulting directed $A_\infty$-subcategory. We denote by $X^k$ the object of $\scrA$ corresponding to $V^k$.

\begin{lemma}
There is a cohomologically full and faithful $A_\infty$-functor
\begin{equation} \label{eq:k-functor}
\scrK: \scrF(E) \longrightarrow \Tw(\scrA)
\end{equation}
with the following property, for any $L$:
\begin{equation} \label{eq:thimbles-floer}
\begin{aligned}
& H^*(\hom_{\Tw(\scrA)}(X^k,\scrK(L))) \iso \mathit{HF}^*(\Delta^k,L), \\
& H^*(\hom_{\Tw(\scrA)}(\scrK(L),X^k)) \iso \mathit{HF}^*(L,\Delta^k).
\end{aligned}
\end{equation}
(Note that even though $\Delta^k \subset E$ is a submanifold with boundary, these Floer cohomology groups are well-defined, by a simple maximum principle argument).
\end{lemma}

This is essentially \cite[Corollary 18.25]{seidel04}. The equality \eqref{eq:thimbles-floer} follows from the way the embedding is constructed in \cite{seidel04}, which we now recall briefly. One defines an $A_\infty$-category $\scrF(\pi)$ associated to the Lefschetz fibration. The objects are (roughly speaking) both exact Lagrangian branes in $E$ and Lefschetz thimbles. There are cohomologically full and faithful embeddings \cite[Propositions 18.13 and 18.14]{seidel04}
\begin{equation} \label{eq:2-emd}
\scrF(E) \longrightarrow \scrF(\pi) \longleftarrow \scrA.
\end{equation}
Consider the right-hand functor. This maps the object $X^k$ to the Lefschetz thimble $\Delta^k$. One shows \cite[Propositions 18.17 and 18.23]{seidel04} that it induces a quasi-equivalence $\Tw(\scrA) \longrightarrow \Tw(\scrF(\pi))$. One obtains \eqref{eq:k-functor} by inverting that quasi-equivalence, and then composing with the left-hand functor in \eqref{eq:2-emd}. In particular, 
\begin{equation}
\begin{aligned}
& H^*(\hom_{\Tw(\scrA)}(X^k,\scrK(L))) \iso H^*(\hom_{\scrF(\pi)}(\Delta^k,L)), \\
& H^*(\hom_{\Tw(\scrA)}(\scrK(L),X^k)) \iso H^*(\hom_{\scrF(\pi)}(L,\Delta^k)).
\end{aligned}
\end{equation}
Technically, $\scrF(\pi)$ is defined in \cite{seidel04} using a $\bZ/2$-symmetry trick, but it is straightforward to see that the outcome is isomorphic to the right hand side of \eqref{eq:thimbles-floer}.

\begin{figure}
\begin{centering}
\begin{picture}(0,0)%
\includegraphics{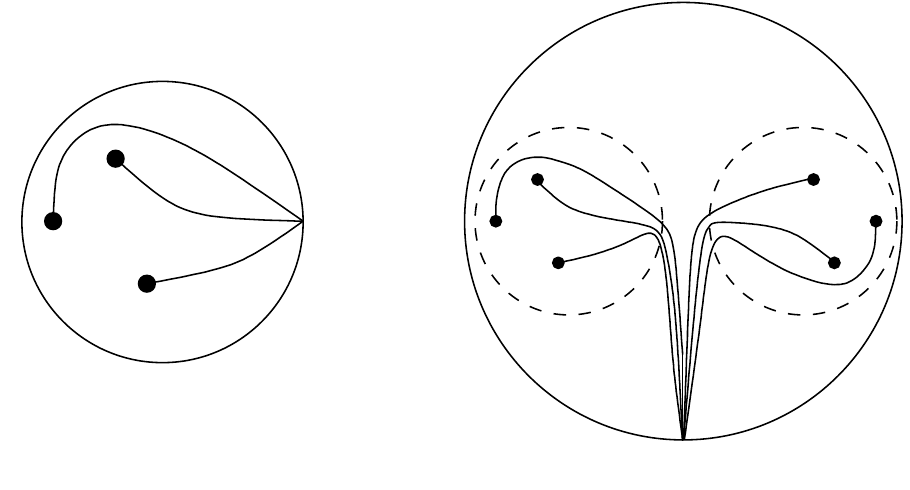}%
\end{picture}%
\setlength{\unitlength}{3947sp}%
\begingroup\makeatletter\ifx\SetFigFont\undefined%
\gdef\SetFigFont#1#2#3#4#5{%
  \reset@font\fontsize{#1}{#2pt}%
  \fontfamily{#3}\fontseries{#4}\fontshape{#5}%
  \selectfont}%
\fi\endgroup%
\begin{picture}(4363,2346)(1096,-1650)
\put(2124,-750){\makebox(0,0)[lb]{\smash{{\SetFigFont{10}{12.0}{\rmdefault}{\mddefault}{\updefault}{\color[rgb]{0,0,0}$\gamma^1$}%
}}}}
\put(1111,-1635){\makebox(0,0)[lb]{\smash{{\SetFigFont{10}{12}{\rmdefault}{\mddefault}{\updefault}{\color[rgb]{0,0,0}The original fibration}%
}}}}
\put(3368,-1635){\makebox(0,0)[lb]{\smash{{\SetFigFont{10}{12}{\rmdefault}{\mddefault}{\updefault}{\color[rgb]{0,0,0}The double branched cover}%
}}}}
\put(3844,-736){\makebox(0,0)[lb]{\smash{{\SetFigFont{10}{12.0}{\rmdefault}{\mddefault}{\updefault}{\color[rgb]{0,0,0}$\tilde{\gamma}^1$}%
}}}}
\put(4673,-163){\makebox(0,0)[lb]{\smash{{\SetFigFont{10}{12.0}{\rmdefault}{\mddefault}{\updefault}{\color[rgb]{0,0,0}$\tilde{\gamma}^{m+1}$}%
}}}}
\end{picture}%
\caption{\label{fig:double-fibre}}
\end{centering}
\end{figure}
An important player in our argument will be the double cover $\tilde{E}$ of $E$ branched along a fibre. This double cover is again the total space of a Lefschetz fibration, in the sense of Setup \ref{th:symplectic-lefschetz}, which we denote by
\begin{equation} \label{eq:tilde-pi}
\tilde{\pi}: \tilde{E} \longrightarrow \tilde{D}.
\end{equation}
Its fibre is the same $F$ as before. From the given basis of vanishing paths, one gets an induced basis $(\tilde{\gamma}^1,\dots,\tilde{\gamma}^{2m})$ for \eqref{eq:tilde-pi}, by the process shown in Figure \ref{fig:double-fibre}. The resulting basis of vanishing cycles consists of two copies of the previous basis:
\begin{equation} \label{eq:tilde-v}
\tilde{V}^k = \begin{cases} V^k & k \leq m, \\ V^{k-m} & k>m. \end{cases}
\end{equation}
In parallel with the previous construction, we have a directed $A_\infty$-category $\tilde{\scrA}$ with objects $(\tilde{X}^1,\dots,\tilde{X}^{2m})$ corresponding to our vanishing cycles. In addition, $\tilde{E}$ contains a collection of Lagrangian spheres $\Sigma^1,\dots,\Sigma^m$, the matching cycles associated to the paths $\sigma^1,\dots,\sigma^m$ from Figure \ref{fig:matching-path}.
\begin{figure}
\begin{centering}
\begin{picture}(0,0)%
\includegraphics{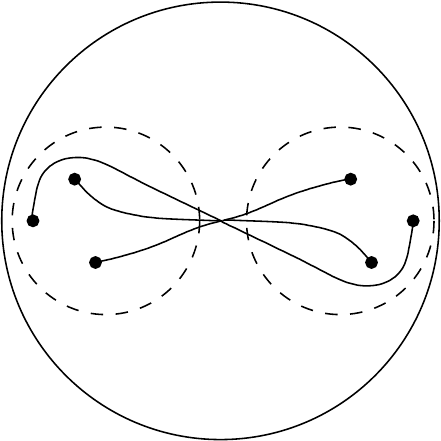}%
\end{picture}%
\setlength{\unitlength}{3947sp}%
\begingroup\makeatletter\ifx\SetFigFont\undefined%
\gdef\SetFigFont#1#2#3#4#5{%
  \reset@font\fontsize{#1}{#2pt}%
  \fontfamily{#3}\fontseries{#4}\fontshape{#5}%
  \selectfont}%
\fi\endgroup%
\begin{picture}(2116,2114)(3318,-1418)
\put(3915,-668){\makebox(0,0)[lb]{\smash{{\SetFigFont{10}{12.0}{\rmdefault}{\mddefault}{\updefault}{\color[rgb]{0,0,0}$\sigma^1$}%
}}}}
\end{picture}%
\caption{\label{fig:matching-path}}
\end{centering}
\end{figure}

Let $D^* \subset D \setminus \partial D$ be a slightly shrunk version of the disc $D$. Then, $E^* = \pi^{-1}(E)$ is again an exact symplectic manifold with corners, so one can consider the Fukaya category $\scrF(E^*)$. When forming the double branched cover $\tilde{E}$, one can take the branch fibre to lie outside $D^*$, and equip $\tilde{E}$ with the pullback of the symplectic form on $E$ plus a form supported in a small neighbourhood of that branch fibre. This means that $E^*$ can be identified with a subset of $\tilde{E}$, in a way which respects the symplectic forms. This identification also respects the other data that enter into the definition of the Fukaya category, namely the one-form primitives of the symplectic form (up to exact one-forms) and the trivialization of the canonical bundle (up to homotopy). More precisely, there are two choices for the inclusion $E^* \hookrightarrow \tilde{E}$, corresponding to the connected components of the preimage of $D^*$ in $\tilde{D}$. Our convention is to choose the component which intersects $\tilde{\gamma}^k$ for $k > m$ (in Figure \ref{fig:double-fibre}, this is the interior of the rightmost dotted circle).

\begin{lemma} \label{th:tilde-embedding}
The inclusions $E^* \hookrightarrow E$, $E^* \hookrightarrow \tilde{E}$ determine $A_\infty$-functors $\scrF(E^*) \rightarrow \scrF(E)$, $\scrF(E^*) \rightarrow \scrF(\tilde{E})$, of which the first one is a quasi-equivalence, and the second one cohomologically full and faithful.
\end{lemma}

The fact that one gets cohomologically full and faithful embeddings is an easy maximum principle argument. In the case of $E^* \hookrightarrow E$, one gets a quasi-equivalence because every exact Lagrangian brane in $E$ can be moved inside $E^*$ by the Liouville flow, which is an exact Lagrangian isotopy (compare the proof of \cite[Proposition 18.13]{seidel04}).

\begin{lemma} \label{th:generators}
If $L$ is an exact Lagrangian brane in $E^*$, there is some $k$ such that $\mathit{HF}^*(L,\Sigma^k) \neq 0$.
\end{lemma}

This is a direct consequence of \cite[Lemma 18.15]{seidel04} (ignoring the issue of $\bZ/2$-equivariance which is important there, but irrelevant here). To outline that argument quickly, the composition of all the Dehn twists $\tau_{\Sigma^k}$ maps $L$ to another Lagrangian submanifold in $\tilde{E}$ which is disjoint from $L$. But if $\mathit{HF}^*(L,\Sigma^k)$ were $0$ for all $k$, one would have $\mathit{HF}^*(L,\tau_{\Sigma^1}\cdots\tau_{\Sigma^m}(L)) \iso \mathit{HF}^*(L,L) \neq 0$ by the long exact sequence from \cite{seidel01}, which is a contradiction.

\begin{lemma} \label{th:matching-cycle}
Consider the analogue of \eqref{eq:k-functor} for the double branched cover, $\tilde{\scrK}: \scrF(\tilde{E}) \rightarrow \Tw(\tilde{\scrA})$. For a suitable choice of grading on $\Sigma^k$, there is a quasi-isomorphism $\tilde{\scrK}(\Sigma^k) \htp S^k$, where 
\begin{equation} \label{eq:matching-cone}
S^k = \mathit{Cone}(\tilde{X}^k \longrightarrow \tilde{X}^{k+m})
\end{equation}
is the mapping cone of a nonzero element of 
\begin{equation}
H^0(\hom_{\tilde{\scrA}}(\tilde{X}^i,\tilde{X}^{k+m})) \iso \mathit{HF}^0(\tilde{V}^k, \tilde{V}^{k+m}) = \mathit{HF}^0(V^k,V^k) \iso H^0(V^k;\bC) = \bC.
\end{equation}
\end{lemma}

This is \cite[Proposition 18.21]{seidel04} combined with the quasi-equivalence $\Tw(\tilde{\scrA}) \htp \Tw(\scrF(\tilde\pi))$ (which, like any $A_\infty$-functor, respects mapping cones up to quasi-isomorphism).

\subsection{Equivariance and Lefschetz fibrations\label{subsec:proofs}}
Continuing in the previous setup, we now impose the assumption that $\scrF(F)$ admits a $\bC^*$-action, given by
\begin{equation} \label{eq:i-f}
\scrI: \scrF(F) \longrightarrow \Tw(\scrC), 
\end{equation}
such that each $\scrI(V^k)$ is quasi-isomorphic to an equivariant twisted complex $C^k$. In particular, one can now replace $\scrA$ with the quasi-isomorphic directed $A_\infty$-subcategory of $\EqTw(\scrC)$ associated to the collection of objects $(C^1,\dots,C^m)$; and this carries a naive $\bC^*$-action. In view of \eqref{eq:k-functor}, it follows that $\scrF(E)$ carries a $\bC^*$-action (in the sense of Definition \ref{th:fuk-has-an-action}). However, it is not clear from this construction whether the action is dilating, and we will therefore take a slightly more roundabout way approach.

Namely, consider \eqref{eq:tilde-pi}, and apply to it the same construction as before. This means that we consider the collection of objects $(\tilde{C}^1,\dots,\tilde{C}^{2m})$ in $\EqTw(\scrC)$ related to $(C^1,\dots,C^m)$ in the same way as in \eqref{eq:tilde-v}, and form the associated directed $A_\infty$-subcategory $\tilde{\scrA}$. This still carries a $\bC^*$-action, and comes with a cohomologically full and faithful $A_\infty$-functor
\begin{equation} \label{eq:tricky}
\scrF(E) \stackrel{\htp}{\longleftarrow} \scrF(E^*) \longrightarrow \scrF(\tilde{E}) 
\stackrel{\tilde{\scrK}}{\longrightarrow} \Tw(\tilde{\scrA}).
\end{equation}
Here, the first two arrows are taken from Lemma \ref{th:tilde-embedding}, and the last one is the analogue of \eqref{eq:k-functor} for the double branched cover. The proof of Theorem \ref{th:gm-picard-lefschetz} is then completed by showing the following:

\begin{lemma}
Suppose that the $\bC^*$-action on $\scrF(F)$ given by \eqref{eq:i-f} is dilating (in the sense of Definition \ref{th:dilating-action}). Then the $\bC^*$-action on $\scrF(E)$ from \eqref{eq:tricky} has the same property.
\end{lemma}

\begin{proof}
Because the naive $\bC^*$-action on $\tilde{\scrA}$ is inherited from that on $\scrC$, the action on 
\begin{equation}
H^0(\hom_{\tilde\scrA}(\tilde{X}^k,\tilde{X}^{k+m})) \iso H^0(\hom_{\EqTw(\scrC)}(C^k,C^k)) \iso \bC
\end{equation}
is trivial. Hence, the mapping cone $S^k$ in \eqref{eq:matching-cone} can be formed with respect to a $\bC^*$-invariant cocycle, and thereby becomes an equivariant twisted complex. Because of Lemma \ref{th:matching-cycle}, this provides an equivariant structure for $\Sigma^k$. There is a canonical isomorphism, parallel to \eqref{eq:endo-compute}, 
\begin{equation} \label{eq:inductive-sphere}
H^*(\hom_{\Tw(\tilde\scrA)}(S^k,S^k)) \iso
\begin{cases} \bC \, [e_C] & \ast = 0, \\
H^{n-1}(\hom_{\EqTw(\scrC)}(C^k,C^k)) & \ast = n, \\
0 & \text{in all other degrees.}
\end{cases} 
\end{equation}
This is compatible with the $\bC^*$-actions, which by assumption means that the $\bC^*$-action on $\mathit{HF}^n(\Sigma^k,\Sigma^k) \iso H^n(\hom_{\Tw(\tilde{\scrA})}(S^k,S^k))$ has weight $1$. In view of Lemmas \ref{th:dilating-generators} and \ref{th:generators}, this implies the desired result.
\end{proof}

\begin{remark}
The outcome of this argument is slightly stronger than the formulation in Theorem \ref{th:gm-picard-lefschetz}: instead of a dilating $\bC^*$-action on $\scrF(F)$, one only needs such an action on the full $A_\infty$-subcategory of $\scrF(F)$ with objects $(V^1,\dots,V^m)$.
\end{remark}

We now turn to the concrete consequences for $q$-intersection numbers. Take the matrix $B_q$ from \eqref{eq:q-b-matrix}, which in algebraic terms means that $B_{q,ij} = C^i \cdot_q C^j$ is the equivariant Mukai pairing on $\EqTw(\scrC)$. Because of the definition of $\tilde{\scrA}$ as a directed subcategory, the equivariant Mukai pairing on 
\begin{equation} \label{eq:2m-k}
K_0^{\bC^*}(\tilde{\scrA}) \iso \bZ[q,q^{-1}]^{2m}
\end{equation}
is described by the matrix $\tilde{A}_q \in \mathit{Mat}(2m \times 2m, \bZ[q,q^{-1}])$ whose entries are
\begin{equation} \label{eq:tilde-a-q}
\tilde{A}_{q,ij} = \begin{cases} 
B_{q,ij} & i<j\leq m, \\
B_{q,(i-m)(j-m)} & m<i<j, \\
B_{q,i(j-m)} & i\leq m <j, \\
1 & i = j, \\
0 & i>j.
\end{cases}
\end{equation}
More succinctly, with $A_q$ as in \eqref{eq:q-a-matrix}, one can rewrite \eqref{eq:tilde-a-q} in block form as
\begin{equation} \label{eq:block-form}
\tilde{A}_q = \begin{pmatrix} A_q & B_q \\ 0 & A_q \end{pmatrix}.
\end{equation}

\begin{lemma}
Let $L$ be an exact Lagrangian brane in $E$, and $C \in \Ob(\Tw(\tilde\scrA))$ its image under \eqref{eq:tricky}. Then for all $k \leq m$,
\begin{align} 
\label{eq:no-left} & H^*(\hom_{\Tw(\tilde{\scrA})}(\tilde{X}^k,C)) = 0, \\
\label{eq:no-right} & H^*(\hom_{\Tw(\tilde{\scrA})}(C,\tilde{X}^k)) = 0.
\end{align}
\end{lemma}

\begin{proof}
Without loss of generality, we can assume that $L$ lies in $E^*$. By applying Lemma \ref{eq:thimbles-floer} to the double branched cover, one sees that the groups in \eqref{eq:no-left}, \eqref{eq:no-right} are the Floer cohomology groups between $L$ and the Lefschetz thimbles $\tilde\Delta^k \subset \tilde{E}$, $k \leq m$. But those thimbles are fibered over paths $\tilde{\gamma}^k$ which are disjoint from $D^* \subset \tilde{D}$ (see Figure \ref{fig:double-fibre}); hence the thimbles are disjoint from $L$.
\end{proof}

Let $L \in \Ob(\scrF(E))$ be an exact Lagrangian brane, and suppose that it carries an equivariant structure given by a quasi-isomorphism between its image in $\mathit{Tw}(\tilde{\scrA})$ and some equivariant twisted complex $\tilde{C}$. We denote the equivariant $K$-theory class of $\tilde{C}$ in \eqref{eq:2m-k} by $\tilde{l}_q$. Recalling that the isomorphism in \eqref{eq:2m-k} is obtained by taking the classes of the $\tilde{X}^k$ as basis vectors, one sees that \eqref{eq:no-right} implies that 
\begin{equation} \label{eq:tilde-l}
\tilde{l}_q = (0,\dots,0,l_q) \in \{(0,\dots,0)\} \times \bZ[q,q^{-1}]^m \subset \bZ[q,q^{-1}]^{2m}.
\end{equation}
With that in mind, \eqref{eq:no-left} shows that \eqref{eq:bq-null} holds.

Take two branes $L_i$ ($i = 0,1$) of the same kind as before, with equivariant structures $\tilde{C}_i$ and equivariant $K$-theory classes $\tilde{l}_{i,q} = (0,\dots,0,l_{i,q})$ as in \eqref{eq:tilde-l}. In view of \eqref{eq:block-form}, the equivariant Mukai pairing between those classes is given by
\begin{equation}
L_0 \cdot_q L_1 = \tilde{C}_0 \cdot_q \tilde{C}_1 = \tilde{l}_{0,q}^*\, \tilde{A}_q\, \tilde{l}_{1,q} = l_{0,q}^*\, A_q\, l_{1,q}.
\end{equation}
This is \eqref{eq:obtain-bullet}. To conclude the proof of Theorem \ref{th:main}, it only remains to show that specializing to $q = 1$ recovers the topological theory from Section \ref{subsec:classical}. We already know from the construction that the $q = 1$ values of $A_q$ and $B_q$ reproduce \eqref{eq:a-matrix} and \eqref{eq:b-matrix}, respectively. If we take $l_q$ and specialize to $q = 1$, we get an $l \in \bZ^n$ whose entries are determined by
\begin{equation}
e^{k,*} \,A \,l = \chi(H^*(\hom_{\Tw(\scrC)}(\tilde{X}^{k+m},\tilde{C}))),
\end{equation}
where $e^k$ is the $i$-th unit vector. Using \eqref{eq:thimbles-floer}, one rewrites this as
\begin{equation}
e^{k,*}\, A \, l = \chi(\mathit{HF}^*(\tilde{\Delta}^{k+m},L)) = \Delta^k \cdot_\pi L.
\end{equation} 
This shows that with respect to the isomorphism $\bZ^m \iso H_\pi$ given by our basis of Lefschetz thimbles, $l$ corresponds to the homology class $[L]$.

\subsection{The examples\label{subsec:the-examples}}
As a stepping-stone, we consider the Milnor fibres $Y$ from Example \ref{th:an-milnor}, assuming that $n \geq 3$. The Lefschetz fibration mentioned in Example \ref{th:simple-milnor} has been analyzed exhaustively in \cite[Section 20]{seidel04}, whose results we recall now (in slightly different notation). Define a directed $A_\infty$-category $\scrC$ with $r+1$ objects and
\begin{equation}
\begin{aligned}
& \hom_{\scrC}(X^i,X^j) = \bC \cdot e^{ji} \oplus \bC \cdot f^{ji} \quad \text{for all $i<j$,} \\
& \text{with } |e^{ji}| = 0, \;\; |f^{ji}| = n-2.
\end{aligned}
\end{equation}
The only nontrivial $A_\infty$-compositions between these generators are
\begin{equation}
\begin{aligned}
& \mu^2_{\scrC}(e^{kj},e^{ji}) = e^{ki}, \\
& \mu^2_{\scrC}(f^{kj},e^{ji}) = f^{ki}, \\
& \mu^2_{\scrC}(e^{kj},f^{ji}) = (-1)^{n-2} f^{ki}.
\end{aligned}
\end{equation}
There is a cohomologically full and faithful embedding
\begin{equation}
\scrI: \scrF(Y) \longrightarrow \Tw(\scrC),
\end{equation}
which takes the Lagrangian spheres $\Sigma^1,\dots,\Sigma^{r+1}$ (with appropriate choices of gradings) from Figure \ref{fig:cyclic} to the twisted complexes
\begin{equation} \label{eq:r+1-twisted}
S^1 = \mathit{Cone}(e^{2,1}), \; \cdots, \; S^r =\mathit{Cone}(e^{r+1,r}), \;
S^{r+1} = \mathit{Cone}(e^{r+1,1})[1].
\end{equation}
%
%
%
Moreover, these spheres (or even any $r$ of them) are split-generators for $\scrF(M)$ \cite[Equation (20.3)]{seidel04}. Concretely, this means that for any exact Lagrangian brane $L \subset M$ there is some $k$ such that $\mathit{HF}^*(L,\Sigma^k) \neq 0$.

Equip $\scrC$ with the naive $\bC^*$-action which has weight $0$ on the $e^{ji}$, and weight $1$ on the $f^{ji}$ (it also acts trivially on the identity endomorphisms of each object, of course). Then \eqref{eq:r+1-twisted} describes equivariant twisted complexes. A computation parallel to \eqref{eq:inductive-sphere} shows that the $\bC^*$-action on $H^{n-1}(\hom_{\Tw(\scrC)}(S^k,S^k))$ has weight $1$ for all $k$. Applying Lemma \ref{th:dilating-generators}, it follows that this equips $\scrF(M)$ with a dilating $\bC^*$-action.

The equivariant Mukai pairing on $K_0^{\bC^*}(\scrC) \iso \bZ[q,q^{-1}]^{r+1}$ is described by the upper-triangular matrix
\begin{equation} \label{eq:milnor-mukai}
\begin{pmatrix} 1 & 1+q(-1)^n & 1+q(-1)^n && \cdots && 1+q(-1)^n \\
& 1 & 1+q(-1)^n &&&& 1+q(-1)^n \\ && 1 &&&& \cdots \\
 &&& \cdots & \\ &&&&& 1 & 1+q(-1)^n \\
&&&&& & 1
\end{pmatrix}
\end{equation}
By construction, the classes of \eqref{eq:r+1-twisted} in this equivariant Grothendieck group are
\begin{equation} \label{eq:sq-classes}
[S^1] = (0,\dots,0,-1,1),\; \cdots,\; [S^r] = (-1,1,0,\dots,0),\; [S^{r+1}] = (1,0,\dots,0,-1).
\end{equation}
Plugging these vectors into \eqref{eq:milnor-mukai} recovers \eqref{eq:deformed-s-matrix}. With this as an input, all the further computations take place entirely within the framework of $q$-intersection numbers.

\begin{example}
Consider the Lefschetz fibration \eqref{eq:fieseler}, whose fibre is the previous $Y$ (still assuming $n \geq 3$). The first vanishing cycle from Figure \ref{fig:ab-cycles} is 
\begin{equation}
V^1 = \tau_{\Sigma^a} \cdots \tau_{\Sigma^2}(\Sigma^1).
\end{equation}
Under the embedding $\scrF(X) \rightarrow \Tw(\scrC)$, the image of $V^1$ is therefore quasi-isomorphic to the equivariant twisted complex $C^1 = T_{S^a} \cdots T_{S^2}(S^1)$.
Combining \eqref{eq:twisted-k-theory}, \eqref{eq:milnor-mukai}, and \eqref{eq:sq-classes}, we get
\begin{equation} \label{eq:k-theory-relation-s}
[C^1] = [S^1] + \cdots + [S^a] \in K_0^{\bC^*}(\scrC).
\end{equation}
For the other $V^i$ and the corresponding equivariant twisted complexes $C^i$, there are similar cyclically rotated expressions, which refine the corresponding homological identities \eqref{eq:s-relation}. For the $q$-intersection numbers, this means that
\begin{equation}
V^i \cdot_q V^j = C^i \cdot_q C^j = \sum_{s=0}^{a-1} \sum_{t=0}^{a-1} S^{i+s} \cdot_q S^{j+t},
\end{equation}
which indeed recovers \eqref{eq:boldb}.
\end{example}
\begin{figure}
\begin{centering}
\begin{picture}(0,0)%
\includegraphics{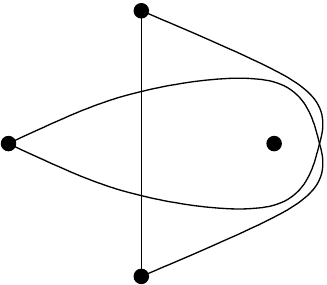}%
\end{picture}%
\setlength{\unitlength}{3355sp}%
\begingroup\makeatletter\ifx\SetFigFont\undefined%
\gdef\SetFigFont#1#2#3#4#5{%
  \reset@font\fontsize{#1}{#2pt}%
  \fontfamily{#3}\fontseries{#4}\fontshape{#5}%
  \selectfont}%
\fi\endgroup%
\begin{picture}(1835,1594)(3703,-1308)
\put(3826,-886){\makebox(0,0)[lb]{\smash{{\SetFigFont{10}{12.0}{\rmdefault}{\mddefault}{\updefault}{\color[rgb]{0,0,0}$V^3$}%
}}}}
\put(4576,-586){\makebox(0,0)[lb]{\smash{{\SetFigFont{10}{12.0}{\rmdefault}{\mddefault}{\updefault}{\color[rgb]{0,0,0}$V^1$}%
}}}}
\put(3901,-286){\makebox(0,0)[lb]{\smash{{\SetFigFont{10}{12.0}{\rmdefault}{\mddefault}{\updefault}{\color[rgb]{0,0,0}$V^2$}%
}}}}
\end{picture}%
\caption{\label{fig:mirrorp2-again}}
\end{centering}
\end{figure}

\begin{example}
In the situation of Examples \ref{th:mirrorp2-example} and \ref{th:q-mirrorp2}, one can redraw Figure \ref{fig:mirrorp2} in a less symmetric way as Figure \ref{fig:mirrorp2-again}. This allows one to write the vanishing cycles in terms of the $\Sigma^i$ from Figure \ref{fig:cyclic} as
\begin{equation}
\begin{aligned}
& V^1 = \tau_{\Sigma^2}(\Sigma^1), \\
& V^2 = \tau_{\Sigma^2}^{-1}\tau_{\Sigma^1}(\Sigma^4), \\
& V^3 = \tau_{\Sigma^3}\tau_{\Sigma^1}(\Sigma^2).
\end{aligned}
\end{equation}
The counterpart of \eqref{eq:k-theory-relation-s} is then
\begin{equation}
\begin{aligned}
& [C^1] = [S^1] + [S^2], \\
& [C^2] = [S^1] - (-1)^n q^{-1} [S^2] + [S^4], \\ 
& [C^3] = -q(-1)^n[S^1] + [S^2] + [S^3].
\end{aligned}
\end{equation}
Plugging those expressions into \eqref{eq:deformed-s-matrix} then yields the matrix \eqref{eq:boldb-2}. 
\end{example}


\begin{thebibliography}{10}

\bibitem{abouzaid-smith13}
M.~Abouzaid and I.~Smith.
\newblock The symplectic arc algebra is formal.
\newblock Preprint arXiv:1311.5535, 2013.

\bibitem{albers10}
P.~Albers.
\newblock Erratum for ``On the extrinsic topology of Lagrangian submanifolds''.
\newblock {\em Int. Math. Res. Not.}, pages 1363--1369, 2010.

%
\bibitem{arnold-gusein-zade-varchenko}
V.~Arnold, S.~Gusein-Zade, and A.~Varchenko.
\newblock {\em Singularities of differentiable maps}.
\newblock Birkh{\"a}user, 1988.

\bibitem{auroux-katzarkov-orlov04}
D.~Auroux, L.~Katzarkov, and D.~Orlov.
\newblock Mirror symmetry for weighted projective planes and their
  noncommutative deformations.
\newblock {\em Ann. of Math.}, 167:867--943, 2008.

\bibitem{bondal-kapranov91}
A.~Bondal and M.~Kapranov.
\newblock Enhanced triangulated categories.
\newblock {\em Math. USSR Sbornik}, 70:93--107, 1991.

\bibitem{caubel-tibar03}
C.~Caubel and M.~Tibar.
\newblock The contact boundary of a complex polynomial.
\newblock {\em Manuscripta Math.}, 111:211--219, 2003.

%
%
\bibitem{dubouloz05}
A.~Dubouloz.
\newblock Danielewski-{F}ieseler surfaces.
\newblock {\em Transf. Groups}, 10:139--162, 2005.

\bibitem{elagin12}
A.~Elagin.
\newblock Cohomological descent theory for a morphism of stacks and for
  equivariant derived categories.
\newblock {\em Mat. Sb.}, 202:31--64, 2011.

\bibitem{givental88}
A.~B. Givental.
\newblock Twisted {P}icard-{L}efschetz formulas.
\newblock {\em Funct. Anal. Appl.}, 22:10--18, 1998.

%

\bibitem{ishii-ueda-uehara10}
A.~Ishii, K.~Ueda, and H.~Uehara.
\newblock Stability conditions on {$A_n$}-singularities.
\newblock {\em J. Differential Geom.}, 84:87--126, 2010.

\bibitem{kervaire57}
M.~Kervaire.
\newblock Relative characteristic classes.
\newblock {\em Amer. J. Math.}, 79:517--558, 1957.

\bibitem{khovanov-seidel98}
M.~Khovanov and P.~Seidel.
\newblock Quivers, {F}loer cohomology, and braid group actions.
\newblock {\em J. Amer. Math. Soc.}, 15:203--271, 2002.

\bibitem{kontsevich98}
M.~Kontsevich.
\newblock Lectures at {ENS} {P}aris.
\newblock {S}et of notes taken by J.\ Bellaiche, J.-F. Dat, I. Marin, G.
  Racinet and H. Randriambololona, 1998.

\bibitem{lekili-maydanskiy12}
Y.~Lekili and M.~Maydanskiy.
\newblock The symplectic topology of some rational homology balls.
\newblock {\em Comment. Math. Helv.}, to appear.

\bibitem{lekili-pascaleff13}
Y.~Lekili and J.~Pascaleff.
\newblock Floer cohomology of {$\mathfrak{g}$}-equivariant {L}agrangian branes.
\newblock Preprint arXiv:1310.8609, 2013.

\bibitem{maydanskiy-seidel09}
M.~Maydanskiy and P.~Seidel.
\newblock {L}efschetz fibrations and exotic symplectic structures on cotangent
  bundles of spheres.
\newblock {\em J. Topology}, 3:157--180, 2010; and Corrigendum, 
{\em J. Topology}, in press.

\bibitem{oh96}
Y.-G. Oh.
\newblock Floer cohomology, spectral sequences, and the {M}aslov class of
  {L}agrangian embeddings.
\newblock {\em Int. Math. Res. Notices}, pages 305--346, 1996.

\bibitem{polishchuk08}
A.~Polishchuk.
\newblock {$K$}-theoretic exceptional collections at roots of unity.
\newblock {\em J. K-Theory}, 7:169--201, 2011.

\bibitem{seidel00b}
P.~Seidel.
\newblock More about vanishing cycles and mutation.
\newblock In: {\em {S}ymplectic {G}eometry and {M}irror {S}ymmetry ({P}roceedings of the 4th {KIAS} Annual International Conference)}, pages 429--465.
\newblock World Scientific, 2001.

\bibitem{seidel01}
P.~Seidel.
\newblock A long exact sequence for symplectic {F}loer cohomology.
\newblock {\em Topology}, 42:1003--1063, 2003.

\bibitem{seidel04}
P.~Seidel.
\newblock {\em {F}ukaya categories and {P}icard-{L}efschetz theory}.
\newblock European Math. Soc., 2008.

\bibitem{seidel12}
P.~Seidel.
\newblock Lagrangian homology spheres in {$(A_m)$} {M}ilnor fibres via
  {$\mathbb{C}^*$}-equivariant ${A}_\infty$-modules.
\newblock {\em Geom. Topol.}, 16:2343--2389, 2012.

\bibitem{seidel13b}
P.~Seidel.
\newblock Lectures on categorical dynamics and symplectic topology.
\newblock Notes, available on the author's homepage, 2013.

\bibitem{seidel13}
P.~Seidel.
\newblock Disjoinable {L}agrangian spheres and dilations.
\newblock {\em Invent. Math.} 197:299--359, 2014.

\bibitem{seidel-solomon10}
P.~Seidel and J.~Solomon.
\newblock Symplectic cohomology and {$q$}-intersection numbers.
\newblock {\em Geom. Funct. Anal.}, 22:443--477, 2012.

\bibitem{tradler01}
T.~Tradler.
\newblock Infinity-inner-products on {A}-infinity-algebras.
\newblock {\em J. Homotopy Related Struct.}, 3:245--271, 2008.

\bibitem{wu55}
W.-T.~Wu.
\newblock
Classes caract\'eristiques et $i$-carr\'es d'une vari\'et\'e.
\newblock {\em C. R. Acad. Sci. Paris}, 230:508--511, 1950.

\end{thebibliography}

\end{document}